\documentclass{article}
\usepackage[a4paper]{geometry}
\usepackage{amsmath,amssymb,amsthm,enumitem,environ,graphicx,tikz}
\usepackage[initials]{amsrefs}

\newtheorem{theorem}{Theorem}[section]
\newtheorem{lemma}[theorem]{Lemma}
\newtheorem{corollary}[theorem]{Corollary}

\theoremstyle{definition}
\newtheorem{definition}[theorem]{Definition}
\newtheorem{problem}[theorem]{Problem}

\numberwithin{equation}{section}
\numberwithin{figure}{section}

\DeclareMathOperator{\tr}{\textnormal{tr}}
\DeclareMathOperator{\Tr}{\textnormal{Tr}}
\let\Im\relax
\DeclareMathOperator{\Im}{Im}
\newtheorem{question}[theorem]{Question}
\makeatletter
\newtheorem*{rep@theorem}{\rep@title}
\newcommand{\newreptheorem}[2]{%
\newenvironment{rep#1}[1]{%
 \def\rep@title{#2 \ref{##1}}%
 \begin{rep@theorem}}%
 {\end{rep@theorem}}}
\makeatother
\newreptheorem{theorem}{Theorem}

\definecolor{col1}{HTML}{D7E3F4} 

\setlist[enumerate]{leftmargin=20pt,itemsep=0pt,topsep=0pt}
\setlist[enumerate,1]{label=(\roman*)}
\setlist[enumerate,2]{label=(\alph*)}

\numberwithin{equation}{section}
\numberwithin{figure}{section}

\usetikzlibrary{calc}
\usetikzlibrary{decorations}
\usepgflibrary{decorations.fractals}
\usetikzlibrary{decorations.fractals}
\usetikzlibrary{decorations.markings}
\usetikzlibrary{arrows.meta}
\tikzstyle{every node}=[circle, draw=black, fill=white, inner sep=0, minimum size=0.2cm] 	
\usetikzlibrary{arrows,shapes,positioning}
\usetikzlibrary{decorations.markings}
\tikzstyle arrowstyle=[scale=1]
\tikzstyle directed=[postaction={decorate,decoration={markings,
		mark=at position 0.5 with {\arrow[arrowstyle]{{latex}}}}}]
\tikzstyle{pdirected}[0.7]=[postaction={decorate,decoration={markings,
  mark=at position #1 with {\arrow[arrowstyle]{{latex}}}}}]

\pgfmathsetmacro{\rad}{1.1}
\pgfmathsetmacro{\Rad}{2}

\def\hgline[#1](#2)(#3:#4:#5)
{
     \pgfmathsetmacro{\thetaone}{#3}
	\pgfmathsetmacro{\thetatwo}{#4}
	\pgfmathsetmacro{\theta}{(\thetaone+\thetatwo)/2}
	\pgfmathsetmacro{\phi}{abs(\thetaone-\thetatwo)/2}
	\pgfmathsetmacro{\close}{less(abs(\phi-90),0.0001)}
	\ifdim \close pt = 1pt
    		\draw[#1] ($(#2)+({#5*cos(\thetaone)},{#5*sin(\thetaone)})$)   --  ($(#2)+({#5*cos(\thetatwo)},{#5*sin(\thetatwo)})$);
	\else
    		\pgfmathsetmacro{\R}{#5*tan(\phi)}
		\pgfmathsetmacro{\startangle}{\thetaone-90}
	     \pgfmathsetmacro{\finishangle}{\startangle+2*\phi-180}
		\draw[#1]($(#2)+({#5*cos(\thetaone)},{#5*sin(\thetaone)})$) arc (\startangle:\finishangle:\R);
	\fi
}

\title{ \vspace{-7ex}\bf \Large Semigroups of isometries of the hyperbolic plane}

\author{Matthew Jacques and Ian Short}

\date{\vspace{-5ex}}

\begin{document}

\maketitle

\begin{abstract}
Motivated by a problem on the dynamics of compositions of plane hyperbolic isometries, we prove several fundamental results on semigroups of isometries, thought of as real M\"obius transformations. We define a semigroup $S$ of M\"obius transformations to be \emph{semidiscrete} if the identity map is not an accumulation point of $S$. We say that $S$ is \emph{inverse free} if it does not contain the identity element. One of our main results states that if $S$ is a semigroup generated by some finite collection $\mathcal{F}$ of M\"obius transformations, then $S$ is semidiscrete and inverse free if and only if every sequence of the form $F_n=f_1\dotsb f_n$, where $f_n\in\mathcal{F}$, converges pointwise on the upper half-plane to a point on the ideal boundary, where convergence is with respect to the chordal metric on the extended complex plane. We fully classify all two-generator semidiscrete semigroups, and include a version of J{\o}rgensen's inequality for semigroups. 

We also prove theorems that have familiar counterparts in the theory of Fuchsian groups. For instance, we prove that every semigroup is one of four standard types: elementary, semidiscrete, dense in the M\"obius group, or composed of transformations that fix some nontrivial subinterval of the extended real line. As a consequence of this theorem, we prove that, with certain minor exceptions, a finitely-generated semigroup $S$ is semidiscrete if and only if every two-generator semigroup contained in $S$ is semidiscrete.

After this we examine the relationship between the size of the `group part' of a semigroup and the intersection of its forward and backward limit sets. In particular, we prove that if $S$ is a finitely-generated nonelementary semigroup, then $S$ is a group if and only if its two limit sets are equal. 

We finish by applying some of our methods to address an open question of Yoccoz.
\end{abstract}

\section{Introduction}

The Denjoy--Wolff theorem in complex analysis states that if $f$ is a holomorphic map of the unit disc $\mathbb{D}$ to itself, then either $f$ is conjugate by an automorphism of $\mathbb{D}$ to a rotation about $0$, or else the sequence of iterates $f^1,f^2,\dotsc$ converges uniformly on compact subsets of $\mathbb{D}$ to a point in $\overline{\mathbb{D}}$, using the Euclidean metric. This paper is motivated by the goal of generalising the Denjoy--Wolff theorem and related results to allow multiple maps. For simplicity, we restrict our attention to automorphisms of $\mathbb{D}$ -- or, equivalently, to automorphisms of the upper half-plane $\mathbb{H}$ -- which are the conformal isometries of $\mathbb{H}$ when $\mathbb{H}$ is endowed with the hyperbolic metric $|dz|/y$, where $z=x+iy$. We find that the resulting dynamical systems are intimately related to certain semigroups of hyperbolic isometries, which we can analyse using methods borrowed from the theory of  Fuchsian groups. In studying semigroups of isometries, we develop material of Fried, Marotta, and Stankewitz \cite{FrMaSt2012} on semigroups of M\"obius transformations, which is one strand among many in the theory of semigroups of rational maps initiated by Hinkkanen and Martin \cite{HiMa1996}.

Our work is also closely related to the theory of $\text{SL}(2,\mathbb{R})$-valued cocycles (see, for example, \cite{Av2011,AvBo2007,BaKaMo2019,BoNa2014,BoNa2015}). In particular, the study of finitely-generated semigroups of hyperbolic isometries is essentially equivalent to the study of locally constant $\text{SL}(2,\mathbb{R})$-valued cocycles over a full shift; see \cite{AvBoYo2010,Yo2004}. The primary focus of those two papers was on classifying uniformly hyperbolic cocycles, and it turns out that this task is closely coupled to our own objectives of analysing semigroups of isometries that satisfy certain discreteness properties. At the end of the paper, we apply the techniques that we have developed to address an open question of Yoccoz from \cite{Yo2004} that was repeated in \cite{AvBoYo2010} (and is stated again in Question~\ref{ccw}, to follow shortly). We answer Yoccoz's question in the negative, suggesting that a reformulation of that question is necessary.  

Let us now set out the notation and terminology necessary to explain our results precisely. The group of conformal isometries of $\mathbb{H}$ is the group $\mathcal{M}$ of real M\"obius transformations $z\mapsto (az+b)/(cz+d)$, where $a,b,c,d\in\mathbb{R}$ and $ad-bc>0$. (Henceforth we refer to such maps simply as \emph{M\"obius transformations}, unless we explicitly state otherwise.) The group $\mathcal{M}$ also acts on the extended real line $\overline{\mathbb{R}}=\mathbb{R}\cup\{\infty\}$ and on the  extended complex plane $\overline{\mathbb{C}}=\mathbb{C}\cup\{\infty\}$. Given a subset $\mathcal{F}$ of $\mathcal{M}$, we define a \emph{composition sequence} generated by $\mathcal{F}$ to be a sequence of composed functions $F_n=f_1\dotsb f_n$, where $f_i\in\mathcal{F}$ for each $i$. For convenience, we define $F_0$ to be the identity map (here and elsewhere). Let us use the notation $(x_n)$ to represent the sequence $x_1,x_2,\dotsc$.  We say that a sequence  $(F_n)$ in $\mathcal{M}$  (not necessarily a composition sequence) is an \emph{escaping sequence} if, for some point $w$ in $\mathbb{H}$, the orbit $(F_n(w))$ does not accumulate in $\mathbb{H}$. This definition does not depend on the choice of the point $w$. An escaping sequence $(F_n)$ is said to \emph{converge ideally} to a point $p$ in $\overline{\mathbb{R}}$ if $(F_n(w))$  converges to $p$ in the chordal metric on  $\overline{\mathbb{C}}$. Once again, the choice of the point $w$ does not affect the definition, in the sense that $(F_n(w))$ converges ideally to $p$ if and only if $(F_n(z))$ converges ideally to $p$, for any point $z$ in $\mathbb{H}$. The following problem encapsulates our chief objective.
 
\begin{problem}\label{kqn}
Classify those finite subsets $\mathcal{F}$ of $\mathcal{M}$ with the property that every composition sequence generated by $\mathcal{F}$ converges ideally. 
\end{problem}

We shall see that there is a close connection between Problem~\ref{kqn} and the theory of semigroups of M\"obius transformations.  Henceforth, for convenience, we use the term \emph{semigroup} to refer to any nonempty subset of $\mathcal{M}$ that is closed under composition. Of course, `semigroup' has a more general meaning in other contexts. There are many similarities, but also some important differences, between the theory of groups of M\"obius transformations and the theory of semigroups. In the former theory it is common to distinguish between discrete groups (Fuchsian groups) and groups that are not discrete. In semigroup theory, discreteness is not the right notion to use (for our purposes); instead we make the following definition.

\begin{definition}
A semigroup $S$ is \emph{semidiscrete} if the identity map is not an accumulation point of $S$ in the topological group $\mathcal{M}$. We say that $S$ is \emph{inverse free} if the inverse of any member of $S$ is not contained in $S$.
\end{definition}

Notice that $S$ is inverse free if and only if it does not contain the identity map $I$, and $S$ is both semidiscrete \emph{and} inverse free if and only if $I$ does not belong to the closure of $S$ in $\mathcal{M}$. Finally, we say that a semigroup $S$  is \emph{generated} by a subset $T$ of $S$ if each element of $S$ can be expressed as a (finite) composition of elements of $T$. If $S$ is generated by a finite set, then we call $S$ a \emph{finitely-generated} semigroup.

Our first theorem relates Problem~\ref{kqn} to the classification of finitely-generated semidiscrete and inverse free semigroups.

\begin{theorem}\label{cof}
Let $S$ be a semigroup generated by a finite collection $\mathcal{F}$ of M\"obius transformations. The following statements are equivalent: 
\begin{enumerate}
\item\label{aaa} every composition sequence generated by $\mathcal{F}$ is an escaping sequence
\item\label{aab} every composition sequence generated by $\mathcal{F}$ converges ideally
\item\label{aac} $S$ is semidiscrete and inverse free.\qedhere
\end{enumerate}
\end{theorem}

In fact, we prove a stronger result than this theorem, which we state after the following definition. A semigroup $S$ is  \emph{elementary} if it has a finite orbit in its action on $\overline{\mathbb{H}}$ (where $\overline{\mathbb{H}}=\mathbb{H}\cup\overline{\mathbb{R}}$); otherwise $S$ is \emph{nonelementary}. The stronger result is that if $S$ is a  finitely-generated nonelementary semidiscrete semigroup, then the following two statements are equivalent: (a) every composition sequence generated by $S$ that is an escaping sequence converges ideally; (b) $S$ is not a cocompact Fuchsian group. The sense in which this result is stronger than Theorem~\ref{cof} will be made precise in Section~\ref{mnm}. 

Notice that assertion (ii) in Theorem~\ref{cof} is equivalent to the statement \emph{every composition sequence generated by $S$ converges ideally} (and there is a similar statement equivalent to (i)). So it is possible to state the theorem without mention of an explicit generating set.

Theorem~\ref{cof} does not solve Problem~\ref{kqn} in any practical sense; instead it shows that Problem~\ref{kqn} is equivalent to another interesting problem (that of classifying finitely-generated inverse-free semidiscrete semigroups). The next theorem solves this equivalent problem when there are only two generators. Our methods are inspired by those of Gilman and Maskit \cite{Gi1988,Gi1991,Gi1995,GiMa1991}, who described a geometric algorithm for classifying two-generator Fuchsian groups. Classifying two-generator semidiscrete semigroups is a significantly simpler task because the difficult case for Fuchsian groups of two hyperbolic generators with intersecting axes is straightforward for semigroups: such maps always generate a semidiscrete semigroup. 	

Let us denote by $\langle f, g\rangle$ the \emph{semigroup} (not the group) generated by two M\"obius transformations $f$ and $g$ (and we use similar notation for semigroups generated by larger finite sets). If $f$ is hyperbolic, then we write $\alpha_f$ and $\beta_f$ for its attracting and repelling fixed points, respectively (and similarly for $g$). If both $f$ and $g$ are hyperbolic, then we define the cross ratio
\[
C(f,g) = \frac{(\alpha_f-\alpha_g)(\beta_f-\beta_g)}{(\alpha_f-\beta_g)(\beta_f-\alpha_g)},
\]
with the usual conventions about $\infty$. We also define the commutator $[f,g]=fgf^{-1}g^{-1}$. In addition, for any M\"obius transformation $f$ we define $\tr(f)=|a+d|$, where $f(z)=(az+b)/(cz+d)$, with $ad-bc=1$. 

We have one final definition before we can state our first theorem about classifying two-generator semigroups, which is taken from the theory of discrete groups. Let $A$, $B$, $C$, and $D$ be disjoint open intervals in $\overline{\mathbb{R}}$, and let $f$ and $g$ be M\"obius transformations such that $f$ maps the complement of $\overline{A}$ to $B$, and $g$ maps the complement of $\overline{C}$ to $D$. A group generated by maps $f$ and $g$ of this type is called a \emph{Schottky group} of rank two. Such groups are discrete and free. We discuss Schottky groups in more detail later on.

\begin{theorem}\label{piv}
Let $f$ and $g$ be  M\"obius transformations such that $\tr(f)\leq \tr(g)$ and $\langle f, g\rangle$ is nonelementary. Then $\langle f, g\rangle$ is semidiscrete and inverse free if and only if exactly one of the following holds:
\begin{enumerate}
\item\label{ipa} $f$ and $g$ are parabolic, and $fg$ is not elliptic
\item\label{ipb} $f$ is parabolic and $g$ is hyperbolic, and $f^ng$ is not elliptic for any positive integer $n$, or
\item\label{ipc} $f$ and $g$ are hyperbolic, and either  $C(f,g)<1$, or $C(f,g)>1$ and $fg$ is not elliptic and either
\begin{enumerate}
\item\label{ipd} $\tr [f,g]\geq \tr(g)^2-2$, or
\item\label{ipe} $\tr [f,g]< \tr(g)^2-2$ and $\langle f, fg\rangle$ is semidiscrete and inverse free.
\end{enumerate}
\end{enumerate}
Furthermore, if $\langle f,g\rangle$ is semidiscrete and inverse free, then either $f$ and $g$ both map a closed interval strictly inside itself, or else $\langle f,g\rangle$ is contained in a Schottky group of rank two.
\end{theorem}

In fact, we will prove that in case~{\ref{ipb}} we have $\tr[f,g]>2$, and it suffices to check that $f^ng$ is not elliptic for $1\leq n\leq \tr(g)/\sqrt{\tr[f,g]-2}$ (see Theorem~\ref{acd}).

We will see in Section~\ref{ibj} that the conditions in statement~{\ref{ipc}} correspond to precise geometric configurations of the axes of the hyperbolic transformations $f$ and $g$. In that section we will also show that if the pair $\{f,g\}$ is of class~{\ref{ipc}}{\ref{ipe}}, and then we apply the algorithm again to the pair $\{f,fg\}$, and carry on in this fashion, then we will eventually reach a pair that is not of class~{\ref{ipc}}{\ref{ipe}}. Theorem~\ref{piv} can thereby be used as the basis of an algorithm for deciding in finitely many steps whether two M\"obius transformations generate an inverse-free semidiscrete semigroup.

Results of a similar type to those of Theorem~\ref{piv} have been considered from a different perspective by Avila, Bochi, and Yoccoz \cite{AvBoYo2010} in their study of  $\text{SL}(2,\mathbb{R})$-valued linear cocycles. We briefly summarise the connection. Following earlier work of Yoccoz \cite{Yo2004}, the authors of \cite{AvBoYo2010} examine the locus $\mathcal{H}$ of parameters $(A_1,\dots,A_N)$ in $\text{SL}(2,\mathbb{R})^N$ that give rise to cocycles that are uniformly hyperbolic. In the terminology of this paper,  $(A_1,\dots,A_N)$ gives rise to a uniformly hyperbolic cocycle if for any composition sequence $(F_n)$ generated by  $\{A_1,\dots,A_N\}$ (with $A_1,\dots,A_N$ acting as M\"obius transformations), the sequence $\rho(F_n(i),i)$ tends to infinity at least linearly in $n$ (cf. \cite[Proposition~2]{Yo2004}).

In \cite{AvBoYo2010} it is shown  that  $(A_1,\dots,A_N)\in\mathcal{H}$ if and only if there is a nonempty finitely-connected open subset $X$ of $\overline{\mathbb{R}}$ other than $\overline{\mathbb{R}}$ itself such that, if we allow the matrices $A_i$ to act on $\overline{\mathbb{R}}$, then $\overline{A_i(X)}\subset X$ for $i=1,\dots,N$. In the language of Section~\ref{ooq}, to follow, this says that the transformations corresponding to the maps $A_i$ generate a \emph{Schottky semigroup} with a particularly strict contractive property. The authors of \cite{AvBoYo2010} discuss the structure of $\mathcal{H}$, obtaining a  near-complete understanding when $N=2$. Part of their analysis is similar in spirit to the proof of Theorem~\ref{piv}; in particular, they describe an algorithm for classifying uniformly hyperbolic pairs that is close to the algorithm summarised in the preceding paragraph (see \cite[Remark~3.14]{AvBoYo2010}).%

The set $\mathcal{H}$ is open in the topological space $\text{SL}(2,\mathbb{R})^N$, as is the set $\mathcal{E}$ of $N$-tuples $(A_1,\dots,A_N)$ for which the semigroup $\langle A_1,\dots,A_N\rangle$ contains an elliptic element. In \cite[Proposition~6]{Yo2004}, credited to Avila, it is proven that $\overline{\mathcal{E}}=\mathcal{H}^c$ (where $\overline{\mathcal{E}}$ is the closure of $\mathcal{E}$ in $\text{SL}(2,\mathbb{R})^N$ and $\mathcal{H}^c$ is the complement of $\mathcal{H}$ in $\text{SL}(2,\mathbb{R})^N$). To accompany this proposition, Yoccoz asks the following question, repeated in \cite[Question~4]{AvBoYo2010}

\begin{question}\label{ccw}
Is $\overline{\mathcal{H}}=\mathcal{E}^c$?
\end{question}

Question~\ref{ccw} is answered affirmatively when $N=2$ in \cite[Theorem~3.3]{AvBoYo2010}. However, in Section~\ref{conclusion} we answer the question in the negative when $N=4$ by providing a counterexample; in fact, our example can be adjusted to give a negative answer for all values of $N>2$. 
	
Closely related to Theorem~\ref{piv} is the following result, which classifies the two-generator semidiscrete semigroups. Let $\text{ord}(f)$ denote the order of an elliptic M\"obius transformation $f$ of finite order.

\begin{theorem}\label{kad}
Let $f$ and $g$ be M\"obius transformations such that $\tr(f)\leq \tr(g)$ and $\langle f, g\rangle$ is nonelementary.  Then $\langle f, g\rangle$ is semidiscrete if and only if one of the following holds:
\begin{enumerate}
\item\label{rua} $\langle f, g\rangle$ is semidiscrete and inverse free
\item\label{rub} $\langle f, g\rangle$ is a Fuchsian group, or
\item\label{ruc} $f$ is elliptic of finite order,  and $f^ng$ is not elliptic, for $n=0,1,\dots,\textnormal{ord}(f)-1$.\qedhere
\end{enumerate}
\end{theorem}

Theorem~\ref{kad} leads to the following algorithm for deciding whether a pair of M\"obius transformations $f$ and $g$ generate a semidiscrete semigroup. First check whether  $\langle f,g\rangle$ is semidiscrete and inverse free using the algorithm of Theorem~\ref{piv}. If not, then apply Gilman and Maskit's algorithm \cite{GiMa1991} to check whether $f$ and $g$ generate \emph{as a group} a Fuchsian group. We will see later that any pair of maps of type~{\ref{ruc}} are contained in a Fuchsian group, so if our check reveals that $f$ and $g$ do not generate as a group a Fuchsian group, then Theorem~\ref{kad} tells us that $\langle f,g\rangle$ is not semidiscrete.

As a consequence of Theorems~\ref{piv} and~\ref{kad}, we also establish a version of J{\o}rgensen's inequality for semigroups (see Section~\ref{ibj}).

The case in Theorem~\ref{piv} when both transformations are parabolic is particularly relevant to the theory of continued fractions. This case has been isolated from the theorem in the following corollary.

\begin{corollary}\label{cia}
 Let $f(z)=z+\lambda$ and $g(z) = z/(\mu z+1)$. The semigroup generated by $f$ and $g$ is semidiscrete and inverse free if and only if  $\lambda \mu \notin (-4,0]$. 
\end{corollary}

To see how this corollary relates to continued fractions, write $g(z)=1/(\mu+1/z)$, so  
\[
f^{b_1}g^{b_2}f^{b_3}g^{b_4}\dotsb =\lambda b_1+ \cfrac{1}{\mu b_2+\cfrac{1}{\lambda b_3+\cfrac{1}{\raisebox{-1ex}{$\mu b_4+\dotsb$}}}}\, ,
\]
for positive integers $b_1,b_2,\dotsc$. The corollary together with Theorem~\ref{cof} tell us that if $\lambda\mu\notin (-4,0]$, then all continued fractions of the above type converge ideally.

Let us now turn to other results on the structure of semigroups, drawing inspiration from the theory of Fuchsian groups. These results advance our understanding of composition sequences, but perhaps they are of more interest independently, for their own sake.

Any group of M\"obius transformations is elementary, discrete, or dense in $\mathcal{M}$. Our next theorem is a counterpart of this familiar result, for semigroups. It has a similar statement, but with an additional category. Let $J$ be a closed interval in  $\overline{\mathbb{R}}$ (thinking of $\overline{\mathbb{R}}$ metrically as a circle). We do not consider singletons or $\overline{\mathbb{R}}$ itself to be closed intervals (this convention will be retained throughout). We define $\mathcal{M}(J)$ to be the collection of M\"obius transformations that map $J$ within itself. Clearly $\mathcal{M}(J)$ is a semigroup, which is not semidiscrete.

\begin{theorem}\label{ppd}
Let $S$ be a semigroup. Then $S$ is 
\begin{enumerate}
\item\label{zya} elementary
\item\label{zyb} semidiscrete
\item\label{zyc} contained in $\mathcal{M}(J)$, for some closed interval $J$, or
\item\label{zyd} dense in $\mathcal{M}$.\qedhere
\end{enumerate}
\end{theorem}

Theorem~\ref{ppd} represents a significant generalisation of a result of B{\'a}r{\'a}ny, Beardon, and Carne \cite[Theorem~3]{BaBeCa1996}, who proved that any semigroup generated by two noncommuting transformations, one of which is elliptic of infinite order, is dense in $\mathcal{M}$. 

For any real number $a$, we  write $[a,+\infty]$ for the interval $\{x\,:\, x\geq a\}\cup\{\infty\}$, and  $[-\infty,a]$ for the interval $\{x\,:\, x\leq a\}\cup\{\infty\}$.

If $S$ is finitely generated, then all semigroups bar a few that are contained in $\mathcal{M}(J)$ for some interval $J$ (of type~{\ref{zyc}}) are semidiscrete (of type~{\ref{zyb}}). The few finitely-generated semigroups of type~{\ref{zyc}} that are not of type~{\ref{zyb}} include, for example, the semigroup generated by the maps $\sqrt{2}z$, $\tfrac12z$ and $z+1$ (using the obvious shorthand). This semigroup is contained in $\mathcal{M}([0,+\infty])$, but it is not elementary, semidiscrete, or dense in $\mathcal{M}$. In Section~\ref{gaq} we will classify the small collection of finitely-generated semigroups that are not of types~{\ref{zya}},~{\ref{zyb}}, or~{\ref{zyd}}. 

Our next theorem also has a familiar counterpart in the theory of Fuchsian groups. This counterpart theorem says that a nonelementary group of M\"obius transformations is discrete if and only if each two-generator subgroup of the group is discrete (in fact, this theorem is true for groups of complex M\"obius transformations as well). Our theorem only applies to finitely-generated semigroups (there is a comparable result for semigroups that are not finitely generated, but it is weaker, so we postpone discussing it until Section~\ref{gaq}). Unfortunately, there is also a bothersome class of semigroups that we must treat as exceptional cases, which we now describe. Let $S$ be a semigroup that lies in $\mathcal{M}(J)$, for some closed interval $J$. Suppose that the collection of elements of $S$ that fix $J$ as a set forms a nontrivial discrete group. Suppose also that one of the other members of $S$ (outside this discrete group) fixes one of the end points of $J$. In these circumstances we say that $S$ is \emph{exceptional}; otherwise it is  \emph{nonexceptional}. These terms should be treated in a similar way to how we treat the terms \emph{elementary} and \emph{nonelementary} in the theory of  groups or semigroups; that is, exceptional semigroups, like elementary groups or semigroups, are easy to handle, but do not satisfy all the same laws as their more general cousins. We emphasise that it is simple to tell whether a finitely-generated semigroup is exceptional by examining its generating set, as we explain later. 

\begin{theorem}\label{abd}
Any finitely-generated  nonexceptional nonelementary semigroup $S$ is semidiscrete if and only if every two-generator semigroup contained in $S$ is semidiscrete.
\end{theorem}

Our final theorem is about limit sets of semigroups. Limit sets of semigroups of complex M\"obius transformations have been studied before, in \cite{FrMaSt2012}. We define the \emph{forward limit set} $\Lambda^+(S)$ of $S$ to be the set of accumulation points in $\overline{\mathbb{R}}$ of the set $\{g(w)\,:\, g\in S\}$, where $w\in\mathbb{H}$, using the chordal metric on $\overline{\mathbb{C}}$. The forward limit set is independent of the particular choice of the point $w$ in $\mathbb{H}$. Let $S^{-1}=\{g^{-1}\,:\, g\in S\}$, which is also a semigroup. The \emph{backward limit set} $\Lambda^-(S)$ of $S$ is defined to be $\Lambda^+(S^{-1})$. 

Our theorem is based on the intuitive idea that that there should a relationship between the size of the intersection $\Lambda^+(S)\cap\Lambda^-(S)$ and the size of the group $S\cap S^{-1}$. For example, if $\Lambda^+(S)\cap\Lambda^-(S)$ is countable, then the group $S\cap S^{-1}$ is either empty or elementary (because if it is not elementary, then its limit set, which is contained in both $\Lambda^+(S)$ and $\Lambda^-(S)$, is perfect and hence uncountable). The next theorem is about when the intersection $\Lambda^+(S)\cap\Lambda^-(S)$ is large.

\begin{theorem}\label{zkj}
Let $S$ be a finitely-generated semidiscrete semigroup such that $|\Lambda^-(S)|\neq 1$. Then $\Lambda^+(S)=\Lambda^-(S)$ if and only if $S$ is a group.
\end{theorem}

In fact, we prove that if $\Lambda^-(S)\subset \Lambda^+(S)$, then $S$ is a group.


\section{The M\"obius group and semigroups}\label{ciq}


Let us begin by introducing  notation for the M\"obius group, and developing some of the basic theory of semigroups.

The extended complex plane $\overline{\mathbb{C}}$ is a compact metric space when endowed with the chordal metric
\[
\chi(z,w)= \frac{2|z-w|}{\sqrt{1+|z|^2}\sqrt{1+|w|^2}},\quad \chi(z,\infty) = \frac{2}{\sqrt{1+|z|^2}},\quad z,w\neq\infty,
\]
which is the metric inherited by $\overline{\mathbb{C}}$ from the restriction of the three-dimensional Euclidean metric to the unit sphere, after identifying  $\overline{\mathbb{C}}$ with the unit sphere by stereographic projection. We denote the closure of a set $X$ in $\overline{\mathbb{C}}$ by $\overline{X}$. In particular, if $X$ is a subset of the hyperbolic plane $\mathbb{H}$, then $\overline{X}$ may contain points that lie outside~$\mathbb{H}$.

We can equip the M\"obius group $\mathcal{M}$ with the uniform metric $\chi_0$ with respect to $\chi$, which is given by
\[
\chi_0(f,g) =\sup_{z\in\overline{\mathbb{C}}} \chi(f(z),g(z)),\quad f,g\in\mathcal{M}.
\]
The metric space $(\mathcal{M},\chi_0)$ is both a complete metric space and a topological group. The metric $\chi_0$ is right-invariant. We denote the identity map in $\mathcal{M}$ by $I$ throughout.

Recall that a semigroup $S$ is a subset of $\mathcal{M}$ that is closed under composition. Semigroups share a number of basic properties with subgroups of $\mathcal{M}$ -- for example, the conjugate of a semigroup is also a semigroup -- but they also have features not shared by groups. For instance, a semigroup can be inverse free. If $S$ is inverse free, then we can tack the identity map $I$ on to $S$ to form another semigroup, $S\cup\{I\}$.

As we have observed already, the set $S^{-1}$ of inverse elements of a semigroup $S$ is itself a semigroup. The semigroups $S$ and $S^{-1}$ have common characteristics: one is semidiscrete if and only if the other is semidiscrete, one is finitely generated if and only if the other is finitely generated, and so on. From $S$ and $S^{-1}$ we can create one or two new semigroups that lie inside $S$, as follows. If the intersection $S\cap S^{-1}$ is not empty, then it is a group, which we call the \emph{group part} of $S$. If the set $S\setminus S^{-1}$ is not empty, then it is an inverse-free semigroup, which we call the \emph{inverse-free part} of $S$. In fact, the following elementary lemma says that when you compose an element of the inverse-free part of $S$ with any other element of $S$, you obtain another element from the inverse-free part of $S$.

\begin{lemma}\label{kbi}
Let $S$ be a semigroup. If $g,h\in S$, and one of them belongs to $S\setminus S^{-1}$, then $gh\in S\setminus S^{-1}$.
\end{lemma}
\begin{proof}
We prove the contrapositive assertion that if $gh\in S\cap S^{-1}$, then $g,h\in S\cap S^{-1}$. If $gh\in S\cap S^{-1}$, then $f=(gh)^{-1}$ is an element of $S$. Therefore $g^{-1}=hf$ and $h^{-1}=fg$ both belong to $S$, as required.
\end{proof}

Remember that a semigroup $S$ is said to be generated by a subset $\mathcal{F}$ of $\mathcal{M}$ if $S$ consists of finite compositions of elements of $\mathcal{F}$. We call a finitely-generated semigroup $S$ an \emph{$n$-generator semigroup} if the smallest set that generates $S$ has order $n$.

\begin{lemma}\label{can}
Suppose that $S$ is a semigroup generated by a set $\mathcal{F}$. If $S$ has nonempty group part, then that group is generated by a subset of $\mathcal{F}$.
\end{lemma}
\begin{proof}
Let $\mathcal{N}$ denote the elements of $\mathcal{F}$ that belong to the inverse-free part of $S$, and let $\mathcal{G}$ denote the elements of $\mathcal{F}$ that belong to the group part of $S$. By Lemma~\ref{kbi}, a word in the generators $\mathcal{F}$ represents an element of $S\cap S^{-1}$ if and only if all the letters in the word belong to $\mathcal{G}$. Therefore $S\cap S^{-1}$ is generated by $\mathcal{G}$.
\end{proof}

\begin{corollary}\label{aig}
If nonempty, the group part of a finitely-generated semigroup is a finitely-generated group. 
\end{corollary}

In contrast, the inverse-free part of a finitely-generated semigroup need not be finitely generated.  For example, suppose that $f$ and $g$ generate \emph{as a group} a Schottky group $G$, and let $S=\langle f,g,g^{-1}\rangle$, the semigroup generated by $f$, $g$, and $g^{-1}$. Given any integer $n$, observe that it is not possible to write $fg^n$ as a product of two elements of $S\setminus S^{-1}$ (using the facts that $G$ is a free group and words in $f$ and $g$ from $S\setminus S^{-1}$ use only positive powers of $f$). Therefore any generating set for $S\setminus S^{-1}$ must contain every element $fg^n$, so it is infinite.

\section{Semidiscrete semigroups}

In this section we discuss the relationship between being discrete and semidiscrete, and we look at some other properties of semigroups that are equivalent to being semidiscrete. It is helpful to recall that a set $T$ of M\"obius transformations is said to be discrete if it is a discrete subset of  the M\"obius group $\mathcal{M}$, and $T$ is semidiscrete if the identity map $I$ is not an accumulation point of $T$ in $\mathcal{M}$. If $S$ is a semidiscrete semigroup, then its group part and inverse-free part (if nonempty) are also semidiscrete, so, in particular, the group part is a Fuchsian group. 

First we show that semidiscrete and discrete are different for semigroups. One of the simplest examples of a semidiscrete semigroup that is not discrete is the two-generator semigroup $S$ generated by $f(z)=2z$ and $g(z)=\tfrac12 z+1$. This semigroup is not discrete because
\[
g^nf^n(z) = z+2-\frac{1}{2^{n-1}} \to z+2\quad\text{as}\quad n\to\infty.
\]
However, it is semidiscrete because one can easily check that  each element of $S$ other than a positive power of $f$ has the form $z\mapsto 2^nz+b$, where $n\in\mathbb{Z}$ and $b\geq 1$.

The semigroup $S$ is elementary as both $f$ and $g$ fix $\infty$; however, it is easy to modify $S$ to give a nonelementary semidiscrete semigroup that is not discrete. For example, $\langle f,g,h\rangle$, where $h(z)=z/(2z+1)$, is such a semigroup, as one can easily check (or use Theorem~\ref{jjr}).

Now let us examine properties of semigroups that are equivalent to being semidiscrete.  The action of a semigroup $S$ on $\mathbb{H}$ is said to be \emph{properly discontinuous} if for each point $w$ in $\mathbb{H}$ there is a neighbourhood $U$ of $w$ such that $g(U)\cap U\neq \varnothing$ for only finitely many elements $g$ of $S$. When $S$ is a group it is discrete if and only if its action on $\mathbb{H}$ is properly discontinuous, and this is so if and only if the $S$-orbit of any point $w$ in $\mathbb{H}$ does not accumulate anywhere in $\mathbb{H}$. The next theorem is a comparable result for semigroups.

\begin{theorem}\label{lla}
Let $S$ be a semigroup. The following statements are equivalent: 
\begin{enumerate}
\item\label{lka} $S$ is semidiscrete
\item\label{lkb} the action of $S$ on $\mathbb{H}$ is properly discontinuous
\item\label{lkc} the $S$-orbit of any point $w$ in $\mathbb{H}$ does not accumulate at $w$.\qedhere
\end{enumerate}
\end{theorem}

We omit the proof, as it is elementary, and similar to proofs of comparable theorems from the theory of Fuchsian groups, such as those of \cite[Theorems~2.2.1 and~2.2.6]{Ka1992}.

Sometimes in the literature on Fuchsian groups a group $S$ is said to act properly discontinuous on $\mathbb{H}$ if for each compact subset $K$ of $\mathbb{H}$ there are only finitely many elements $g$ of $S$ that satisfy $g(K)\cap K\neq\varnothing$. For groups, this definition is equivalent to the definition we gave earlier; however, for semigroups the two definitions differ. To see this, consider the semigroup $S=\langle f,g\rangle$, where $f(z)=2z$ and $g(z)=\tfrac12 z+1$, discussed near the start of this section. Let $K=\{t+i:t\in[0,2]\}$, a compact subset of $\mathbb{H}$. Then $g^nf^n(K)\cap K\neq \varnothing$ for all positive integers $n$. On the other hand, we have seen that $S$ is semidiscrete, so the action of $S$ on $\mathbb{H}$ is properly discontinuous, by Theorem~\ref{lla}.

The action of a semigroup $S$ on $\mathbb{H}$ is said to be \emph{strongly discontinuous} if for each point $w$ in $\mathbb{H}$ there is a neighbourhood $U$ of $w$ such that $g(U)\cap U= \varnothing$ for every element $g$ of $S$. This definition is close to the definition of a discontinuous action in the theory of Fuchsian groups, but in that theory the intersection $g(U)\cap U$ is empty for every element $g$ of $S$ \emph{except} the identity element.

\begin{theorem}
Let $S$ be a semigroup. The following statements are equivalent: 
\begin{enumerate}
\item\label{rka} $S$ is semidiscrete and inverse free
\item\label{rkb} the action of $S$ on $\mathbb{H}$ is strongly discontinuous
\item\label{rkc} the $S$-orbit of any point $w$ in $\mathbb{H}$ stays a positive distance away from $w$.\qedhere
\end{enumerate}
\end{theorem}

The proof (which we omit) is similar to that of Theorem~\ref{lla} and similar to proofs of well-known results for Fuchsian groups, such as \cite[Theorems~2.2.1 and~2.2.6]{Ka1992}.

\section{Composition sequences}

Let us move now from semigroups to composition sequences and prove the first part -- the easy part -- of Theorem~\ref{cof}. We remind the reader that a sequence $(G_n)$ of M\"obius transformations is an escaping sequence if, for some point $w$ in $\mathbb{H}$, the orbit $(G_n(w))$ has no accumulation points in the metric space $(\mathbb{H},\rho)$, where $\rho$ denotes the hyperbolic metric. There are many equivalent ways of describing escaping sequences, one of which is captured in the following standard lemma (we omit the elementary proof).

\begin{lemma}\label{cii}
A sequence of M\"obius transformations is an escaping sequence if and only if it does not contain a subsequence that converges uniformly to a M\"obius transformation. 
\end{lemma}

Using the uniform metric $\chi_0$ introduced in Section~\ref{ciq}, we can recast the statement that a sequence $(G_n)$ of M\"obius transformations converges uniformly to another M\"obius transformation $G$ as $\chi_0(G_n,G)\to 0$ as $n\to\infty$.

A significant part of this paper is about the relationship between composition sequences and semigroups. The key to this relationship is the following simple theorem, which gives the equivalence of~{\ref{aaa}} and~{\ref{aac}} in Theorem~\ref{cof}. (Notice, however, that the theorem below applies to all semigroups, and Theorem~\ref{cof} applies only to finitely-generated semigroups.) 

\begin{theorem}\label{acj}
A semigroup $S$ generated by a collection $\mathcal{F}$ of M\"obius transformations is semidiscrete and inverse free if and only if every composition sequence generated by $\mathcal{F}$ is an escaping sequence.
\end{theorem}
\begin{proof}
Suppose first that $S$ is not both semidiscrete and inverse free. Then either the identity map $I$ is an accumulation point of $S$ in $\mathcal{M}$ or else it belongs to $S$. In short, $I\in \overline{S}$. Accordingly, there is a sequence $(g_n)$ in $S$ that converges uniformly to $I$. By restricting to a subsequence of $(g_n)$, we can assume that $\sum \chi_0(g_n,I)<+\infty$. 

Define $G_n=g_1\dotsb g_n$, for $n=1,2,\dotsc$. Using right-invariance of $\chi_0$, we see that
\[
\chi_0(G_n^{-1},G_{n-1}^{-1})= \chi_0(G_n^{-1}G_n,G_{n-1}^{-1}G_n)=\chi_0(I,g_n).
\]
Therefore $\sum \chi_0(G_n^{-1},G_{n-1}^{-1})<+\infty$, which implies that $(G_n^{-1})$ is a Cauchy sequence.  Hence $(G_n)$ is a Cauchy sequence too, and Lemma~\ref{cii} tells us that it is not an escaping sequence. By writing each map $g_n$ as a composition of elements of $\mathcal{F}$ we obtain a composition sequence generated by $\mathcal{F}$ that is not an escaping sequence.

Conversely, suppose that there are maps $g_n$ in $\mathcal{F}$ such that the composition sequence $G_n=g_1\dotsb g_n$ is not an escaping sequence. Then, by Lemma~\ref{cii}, there is a subsequence $(G_{n_i})$ that converges uniformly to a M\"obius transformation $G$. It follows that $G_{n_{i-1}}^{-1}G_{n_i}\to I$ as $n_i\to \infty$. But
\[
G_{n_{i-1}}^{-1}G_{n_i} = (g_1\dotsb g_{n_{i-1}})^{-1}(g_1\dotsb g_{n_i})= g_{n_{i-1}+1}\dotsb g_{n_i},
\]
so we see that $I\in\overline{S}$. Therefore $S$ is not both semidiscrete and inverse free.
\end{proof}

We say that a sequence of M\"obius transformations is \emph{discrete} if the set of transformations that make up the sequence is a discrete subset of $\mathcal{M}$. Although escaping sequences are all discrete, by definition, the converse does not hold; for example, the trivial sequence $I,I,I,\dotsc$ is discrete, but it is not an escaping sequence.

The preceding theorem gave a condition in terms of composition sequences for a semigroup to be both semidiscrete and inverse free. The next, similar, theorem gives a condition in terms of composition sequences for a semigroup to be semidiscrete. The proof is similar, so we only sketch the details.

\begin{theorem}\label{hqj}
A semigroup $S$ generated by a collection $\mathcal{F}$ of M\"obius transformations is semidiscrete if and only if every composition sequence generated by $\mathcal{F}$ is discrete.
\end{theorem}
\begin{proof}
Suppose first that $S$ is not semidiscrete. Then there is a sequence of \emph{distinct} transformations $(g_n)$ from $S$ that converges uniformly to $I$. As before, we define $G_n=g_1\dotsb g_n$, and, provided that $(g_n)$ is chosen to converge to $I$ sufficiently quickly, we know that $(G_n)$ converges uniformly to a map $G$. Because the maps $g_n$ are distinct, it follows that $G_n\neq G$ for infinitely many positive integers $n$. Hence $(G_n)$ is not discrete. By expressing each map $g_n$ in terms of the generators $\mathcal{F}$ we can obtain a composition sequence generated by $\mathcal{F}$ that is not discrete. 

Conversely, suppose there is a composition sequence $G_n=g_1\dotsb g_n$, with $g_n\in\mathcal{F}$, that is not discrete. Choose a subsequence $(G_{n_i})$ of distinct maps that converges uniformly to a map $G$. Then $G_{n_{i-1}}^{-1}G_{n_i}\to I$ as $n_i\to\infty$, where $G_{n_{i-1}}^{-1}G_{n_i}\in S\setminus\{I\}$, so $S$ is not semidiscrete.
 \end{proof}

\section{Covering regions}

In this section we introduce a concept called a covering region, which plays a role for semidiscrete semigroups somewhat similar to the role played by fundamental regions for discrete groups.

A \emph{covering region} for a semigroup $S$ is a closed subset $D$ of $\mathbb{H}$ with nonempty interior such that 
\[
\bigcup_{g\in S} g(D)=\mathbb{H}.
\]
Of course, $\mathbb{H}$ itself is a covering region for $S$. Let us denote the interior of a set $X$ by $X^\circ$. We say that a covering region $D$ is a \emph{fundamental region} for $S$ if it satisfies the additional property $D^\circ\cap g(D^\circ)=\varnothing$ whenever $g$ is a nonidentity element of $S$. This definition coincides with the usual definition of a fundamental region when $S$ is a Fuchsian group.

The next theorem due to Bell \cite[Theorem~3]{Be1997} says that if a semigroup has a fundamental region with dense interior, then it is in fact a Fuchsian group. We include this pleasing theorem and its proof (although we do not use it again) because \cite{Be1997} is difficult to obtain, and in any case the proof given there is overcomplicated.

\begin{theorem}\label{oop}
Suppose that  a semigroup $S$ has a fundamental region $D$ for which $\overline{D^\circ}=D$. Then $S$ is a Fuchsian group with fundamental region $D$.
\end{theorem}
\begin{proof}
Let $g\in S$. Choose an element $h$ of $S$ such that $g^{-1}(D^\circ)\cap h(D)\neq\varnothing$. Hence $g^{-1}(D^\circ)\cap h(D^\circ)\neq\varnothing$, so $D^\circ \cap gh(D^\circ)\neq\varnothing$. Therefore $gh=I$, and $h=g^{-1}$. It follows that $S$ is closed under taking inverses, so it is a group with fundamental region $D$ (and it must be a Fuchsian group because it has a fundamental region).
\end{proof}

The assumption that $\overline{D^\circ}=D$ cannot be dropped from this theorem, as we can see by considering the inverse-free semigroup $S=\{\lambda z\,:\,\lambda\geq 2\}$ with fundamental region 
\[
D=\{z\in\mathbb{H}\,:\, \text{$1\leq |z|\leq 2$ or $|z|=1/2^n$, $n=1,2,\dotsc$}\}.
\]
The assumption that $\overline{D^\circ}=D$ \emph{can} be dropped from Theorem~\ref{oop}, however, if we assume that $S$ is countable. To see this, let $g\in S$ and consider the set $g^{-1}(\overline{D^\circ})$. This set has nonempty interior, so it is not covered by the countable union $\bigcup_{h\in S} h(\partial D)$ of nowhere-dense sets, by the Baire category theorem. Since $\bigcup_{h\in S} h(D)=\mathbb{H}$ there must be an element $h$ of $S$ such that $g^{-1}(\overline{D^\circ})\cap h(D^\circ)\neq\varnothing$. Then $\overline{D^\circ}\cap gh(D^\circ)\neq\varnothing$, so $D^\circ\cap gh(D^\circ)\neq\varnothing$. Therefore $gh=I$, and $h=g^{-1}$. As before, we deduce that $S$ is a group with fundamental region $D$. 

Let us now focus on covering regions that are not necessarily fundamental regions. In the following lemma we refer to the inverse-free part of a semigroup $S$, which, as you may recall, is the set $S\setminus S^{-1}$, which if nonempty is itself a semigroup.

\begin{lemma}\label{kwj}
If $D$ is a covering region for a semigroup $S$ that is not a group, then $D$ is also a covering region for the inverse-free part of $S$. 
\end{lemma}
\begin{proof}
 Choose any point $p$ in $\mathbb{H}$. Let $f\in S\setminus S^{-1}$. Then there is an element $g$ of $S$ such that $f^{-1}(p)\in g(D)$. Hence $p\in fg(D)$. By Lemma~\ref{kbi}, $fg\in S\setminus S^{-1}$. As $p$ was chosen arbitrarily, we see that $D$ is a covering region for $S\setminus S^{-1}$. 
\end{proof}

In the next theorem, we denote by $\rho$ the hyperbolic metric on $\mathbb{H}$. Recall that a cocompact Fuchsian group is a Fuchsian group $G$ for which the quotient space $\mathbb{H}/G$ is compact.

\begin{theorem}\label{poq}
Any semidiscrete semigroup that has a bounded covering region is a cocompact Fuchsian group.
\end{theorem}
\begin{proof}
Let $S$ be a semidiscrete semigroup. Choose any point $w$ in $\mathbb{H}$. As $S$ has a bounded covering region,  we can choose a suitably large open disc $D$ in the hyperbolic metric, with radius $r$ and centre $w$, such that 
\[
\bigcup_{g\in S} g(D) = \mathbb{H}. 
\]
As $\partial D$ is compact, there is  a finite subset $T$ of $S$ such that 
\[
\partial D \subset \bigcup_{g\in T} g(D). 
\]
Let us choose $T$ to be a minimal set with this property, so $g(D)\cap \partial D\neq\varnothing$ and $\rho(g(w),w)\leq 2r$ for $g \in T$. We construct a composition sequence $G_n=g_1\dotsb g_n$ generated by $T$ as follows. Choose $g_1$ arbitrarily. Now suppose that $n>1$. 
\begin{enumerate}
\item[(i)] If $w\in G_{n-1}(\overline{D})$, then choose $g_n$ arbitrarily from $T$.
\item[(ii)] Otherwise, choose $g_n$ from $T$ so that $\rho(G_n(\overline{D}),w)$ is minimised.
\end{enumerate}
Let $\rho_n=\rho(G_n(\overline{D}),w)$. If $g_n$ is chosen according to (i), then $\rho(G_{n-1}(w),w)\leq r$, so
\[
\rho(G_n(w),w) \leq \rho(G_n(w),G_{n-1}(w))+\rho(G_{n-1}(w),w)= \rho(g_n(w),w)+\rho(G_{n-1}(w),w)\leq 3r.
\]
Hence $\rho_n\leq 3r$. Suppose now that $g_n$ is chosen according to (ii). Notice that the collection $\{G_{n-1}g(D)\,:\, g\in T\}$ covers $G_{n-1}(\partial D)$. Therefore
\[
\rho(G_{n-1}(\overline{D}),w)=\rho(G_{n-1}(\partial D),w)\geq \rho(G_n(\overline{D}),w);
\]
that is, $\rho_{n-1}\geq \rho_n$. 

We deduce that the sequence $(\rho_n)$ is bounded, so $(G_n)$ is not an escaping sequence. Theorem~\ref{acj} now tells us that $S$ is not inverse free. We know from Lemma~\ref{kwj} that the inverse-free part of $S$, if nonempty, also has $D$ as a covering region. The argument above tells us that the inverse-free part of $S$, namely $S\setminus S^{-1}$, must therefore be empty. It follows that $S$ is a group, which is semidiscrete -- a Fuchsian group.

Finally, $S$ is a cocompact Fuchsian group because $\mathbb{H}/S$ is the image of the compact set $\overline{D}$ under the quotient map $\mathbb{H}\to \mathbb{H}/S$.
\end{proof}

Given a semidiscrete semigroup $S$, and a point $w$ in $\mathbb{H}$ that is not fixed by any nonidentity element of $S$, we define the \emph{Dirichlet region} for $S$ centred at $w$ to be the closed, convex set 
\[
D_w(S) = \{z\in\mathbb{H}\,:\, \rho(z,w)\leq \rho(z,g(w))\text{ for all $g$ in $S\setminus\{I\}$}\}.
\]
This is the same definition of a Dirichlet region as that used in the theory of Fuchsian groups. We can always find a point $w$ not fixed by any nonidentity element of $S$ because the set of elliptic elements in $S$ generates a discrete group, which must be countable.

\begin{theorem}
Let $S$ be a semidiscrete semigroup and let $w$ be a point in $\mathbb{H}$ that is not fixed by any nonidentity element of $S$. Then $D_w(S)$ is a covering region for $S$.
\end{theorem}
\begin{proof}
The set $D_w(S)$ is closed because it is an intersection of closed sets. It has nonempty interior because, by Theorem~\ref{lla}, the orbit of $w$ under $S$ does not accumulate at $w$. To verify the covering property, suppose, in order to reach a contradiction, that there is a point $z$ in $\mathbb{H}$ that is not contained in $h(D_w(S))$ for any element $h$ of $S$. Then $h^{-1}(z)\notin D_w(S)$, which implies that there is a map $g$ in $S$ that satisfies $\rho(h^{-1}(z),w)>\rho(h^{-1}(z),g(w))$. That is,
\begin{equation}\label{zka}
\text{for all $h\in S$ there exists $g\in S$ such that $\rho(z,h(w))>\rho(z,hg(w))$.}
\end{equation}
We will now define a composition sequence $F_n=f_1\dotsb f_n$, where $f_i\in S$, recursively. Let $f_1=I$. If $f_1,\dots,f_{n-1}$ have been defined, then, using \eqref{zka}, we let $f_n$ be an element of $S$ that satisfies $\rho(z,F_{n-1}(w))>\rho(z,F_{n-1}f_n(w))$. The resulting composition sequence $(F_n)$ satisfies
\[
\rho(z,w) > \rho(z,F_1(w)) > \rho(z,F_2(w)) > \dotsb.
\]
As $S$ is semidiscrete, Theorem~\ref{hqj} tells us that the composition sequence $(F_n)$ is discrete. But the sequence $(F_n(w))$ is bounded, so it must have a constant subsequence. This contradicts the  strict inequalities above. Thus, contrary to our earlier assumption, 
\[
z\in \bigcup_{h\in S} h(D_w(S)),
\]
so $D_w(S)$ is a covering region for $S$.
\end{proof}

\section{Limit sets}

Every Fuchsian group has a limit set associated to it, which is a subset of $\overline{\mathbb{R}}$. In contrast, a semigroup has two limit sets, a forward and a backward limit set, as we have seen already. To recap, the forward limit set $\Lambda^+(S)$ of a semigroup $S$ is the set of accumulation points in $\overline{\mathbb{R}}$ of the orbit of a point $w$ of $\mathbb{H}$ under $S$. That is, $x\in \Lambda^+(S)$ if and only if there is a sequence $(g_n)$ in $S$ such that $g_n(w)\to x$ (in the chordal metric) as $n\to\infty$. The forward limit set is forward invariant under elements of $S$, in the sense that $h(\Lambda^+(S))\subset \Lambda^+(S)$ whenever $h\in S$. The backward limit set $\Lambda^-(S)$ of $S$ is equal to $\Lambda^+(S^{-1})$, the forward limit set of $S^{-1}$. The backward limit set is backward invariant under elements of $S$, in the sense that $h^{-1}(\Lambda^-(S))\subset \Lambda^-(S)$. Fried, Marotta and Stankewitz studied properties of limit sets in \cite{FrMaSt2012}; because of their interest in complex dynamics, they call $\Lambda^-(S)$ the \emph{Julia set} of $S$. They prove the following theorem (see \cite[Theorem~2.4, Proposition~2.6, and Remark~2.20]{FrMaSt2012}), which has a well known counterpart in the theory of Fuchsian groups.

\begin{theorem}\label{aiu}
Let $S$ be a semigroup that contains hyperbolic elements. Then $\Lambda^-(S)$ is 
\begin{enumerate}
\item\label{zfa} the smallest closed set in $\overline{\mathbb{R}}$ containing all repelling fixed points of hyperbolic elements of $S$
\item\label{zfb} the complement in $\overline{\mathbb{R}}$ of the largest open set on which $S$ is a normal family.
\end{enumerate}
Furthermore, if $S$ is nonelementary, then $\Lambda^-(S)$ is a perfect set and is therefore uncountable.
\end{theorem}

We include statement~{\ref{zfb}} to justify the terminology \emph{Julia set} for $\Lambda^-(S)$; we do not use this observation again.  An immediate consequence of statement~{\ref{zfa}} and the final assertion of the theorem is that $\Lambda^+(S)$ is the smallest closed set containing the attracting fixed points of hyperbolic elements of $S$, and it is uncountable. 

Next we define some important subsets of the forward and backward limit sets, which again have familiar counterparts in the theory of Fuchsian groups. Let $x\in\overline{\mathbb{R}}$. Choose any hyperbolic geodesic segment $\gamma$ that has  one end point $x$ and the other end point in $\mathbb{H}$, and let $w\in\mathbb{H}$. Then $x$ is a \emph{forward conical limit point} of $S$ if there is sequence $(g_n)$ in $S$, and a positive constant $k$, such that  
\begin{enumerate}
\item[(i)] $g_n(w)\to x$ as $n\to\infty$ (in the chordal metric)
\item[(ii)] $\rho(\gamma, g_n(w))<k$.
\end{enumerate}
This definition does not depend on the choice of $\gamma$ or $w$. The \emph{forward conical limit set} $\Lambda_c^+(S)$ of $S$ is the collection of all forward conical limit points of $S$. A \emph{backward conical limit point} of $S$ is a forward conical limit set of $S^{-1}$, and the \emph{backward conical limit set} $\Lambda_c^-(S)$ of $S$ is  $\Lambda_c^+(S^{-1})$. 

We also need another class of limit sets. Before we introduce this new class, we recall that a \emph{horocycle} is a circle in $\overline{\mathbb{C}}$ that is contained almost entirely in $\mathbb{H}$ except for one point, which lies on the ideal boundary $\overline{\mathbb{R}}$. We say that the horocycle is \emph{based} at the point on the ideal boundary. For example, given any positive number $t$, the set $\{z\in\mathbb{H}\,:\, \Im z=t\}\cup\{\infty\}$ is a horocycle based at $\infty$. A \emph{horodisc} is the `inside' of a horocycle: the component of the complement of a horocycle that lies wholly within $\mathbb{H}$.  

We define the \emph{forward horocyclic limit set} $\Lambda_h^+(S)$ of $S$ to consist of those points $x$ in $\Lambda^+(S)$ such that, for any point $w$ in $\mathbb{H}$, the orbit $S(w)$ meets every horodisc based at $x$. And of course the \emph{backward horocyclic limit set} $\Lambda_h^-(S)$ of $S$ is defined to be $\Lambda_h^+(S^{-1})$. Clearly we have the inclusions
\[
\Lambda_c^+(S)\subset \Lambda_h^+(S)\subset \Lambda^+(S)\quad\text{and}\quad \Lambda_c^-(S)\subset \Lambda_h^-(S)\subset \Lambda^-(S).
\]
Conical and horocyclic limit sets have important roles in the theory of Fuchsian groups, so it is no surprise that the sets introduced here play an important part in the study of semigroups.

Suppose now that $(F_n)$ is any sequence of M\"obius transformations. We can define $\Lambda^+(F_n)$ and $\Lambda^-(F_n)$, the forward and backward limit sets of $(F_n)$, in a similar way to how we have defined these limit sets for semigroups. Theorem~\ref{aiu}{\ref{zfa}} fails when we use sequences instead of semigroups, but Theorem~\ref{aiu}{\ref{zfb}} remains true. We can also define $\Lambda_c^+(F_n)$ and $\Lambda_c^-(F_n)$, the forward and backward conical limit sets of $(F_n)$, in much the same way as before. The backward conical limit set features in the theory of continued fractions because of the following theorem of Aebischer \cite[Theorem~5.2]{Ae1990}.

\begin{theorem}\label{uun}
Let $(G_n)$ be an escaping sequence, and let $w\in \mathbb{H}$. Then, for each point $x$ in $\overline{\mathbb{R}}$, we have $\chi(G_n(w),G_n(x))\to 0$ as $n\to\infty$ if and only if $x\notin \Lambda_c^-(G_n)$. 
\end{theorem}

It is straightforward to prove this theorem using hyperbolic geometry; see, for example, \cite[Proposition~3.3]{CrSh2007}. We also have the following lemma, proven in \cite[Lemma~3.4]{CrSh2007}.

\begin{lemma}\label{kka}
Suppose that $x$ is a backward conical limit point of an escaping sequence $(G_n)$. Then there is a sequence of positive integers $n_1,n_2,\dotsc$ and two distinct points $p$ and $q$ in $\overline{\mathbb{R}}$ such that $G_{n_i}(x)\to p$ and $G_{n_i}(z)\to q$ for every point $z$ in $\overline{\mathbb{R}}$ other than $x$. 
\end{lemma}

The remainder of this section is concerned with proving that if the forward horocyclic limit set of a semidiscrete semigroup is the whole of $\overline{\mathbb{R}}$, then every Dirichlet region of $S$ is bounded.

Given $x$ in $\overline{\mathbb{R}}$ and $w$ in $\mathbb{H}$, let $H_x(w)$ denote the horodisc that is based at $x$ and whose boundary passes through $w$.  So, for example, $H_\infty(w)=\{z\in\mathbb{H}\,:\,\Im z>\Im w\}$, where $\Im z$ and $\Im w$ denote the imaginary parts of $z$ and $w$, respectively. Let $\gamma_z=\{z+it\,:t\geq 0\}$, the vertical geodesic segment from $z$ to $\infty$. Then $\overline{\gamma_z}=\gamma_z\cup\{\infty\}$. Given distinct points $u$ and $v$ in $\mathbb{H}$, we define 
\[
K(u,v)=\{z\in\mathbb{H}\,:\, \rho(z,u)\leq\rho(z,v)\}.
\]
 Note that $\overline{K(u,v)}$ is the closure of $K(u,v)$ in $\overline{\mathbb{C}}$ (it is not the same as $K(u,v)$).

The next lemma is an elementary exercise in hyperbolic geometry.

\begin{lemma}\label{iwg}
Suppose that $u$ and $v$ are distinct points in $\mathbb{H}$ and $\Im v\leq \Im u$. Then $\gamma_u\subset K(u,v)$. If  $\Im v<\Im u$, then $\infty\notin \overline{K(v,u)}$.
\end{lemma}

We use the lemma to prove the following theorem, which helps us relate Dirichlet regions to horocyclic limit sets.

\begin{theorem}\label{aaw}
Let $S$ be a semidiscrete semigroup, and let $x\in \overline{\mathbb{R}}$ and $w\in\mathbb{H}$, where $w$ is not fixed by any nonidentity element of $S$. Then $x\in \overline{D_w(S)}$ if and only if the $S$-orbit of $w$ does not meet $H_x(w)$.
\end{theorem}
\begin{proof}
By conjugating, we can assume that $x=\infty$ and $w=i$. Suppose first that the $S$-orbit of $i$ meets $H_\infty(i)$. Then there is an element $g$ of $S$ such that $g(i)\in H_\infty(i)$, so $\Im g(i)>1$. By the second assertion of Lemma~\ref{iwg}, $\infty\notin\overline{K(i,g(i))}$. Since $\overline{D_i(S)}$ is contained in $\overline{K(i,g(i))}$, we see that $\infty\notin \overline{D_i(S)}$.

Conversely, suppose that $g(i)\notin H_\infty(i)$ for every nonidentity element $g$ of $S$. Then, for each map $g$, the first assertion of Lemma~\ref{iwg} tells us that  $\gamma_i\subset K(i,g(i))$. Therefore $\gamma_i\subset D_i(S)$, so $\infty\in \overline{D_i(S)}$.
\end{proof}

An immediate consequence of the theorem is the following important corollary.

\begin{corollary}\label{waa}
Let $S$ be a semidiscrete semigroup. If $\Lambda_h^+(S)=\overline{\mathbb{R}}$, then $S$ is a cocompact Fuchsian group.
\end{corollary}
\begin{proof}
Let $D_w(S)$ be a Dirichlet region for $S$. We are given that $\Lambda_h^+(S)=\overline{\mathbb{R}}$, so the orbit $S(w)$ meets every horodisc in $\mathbb{H}$. Theorem~\ref{aaw} now tells us that  $\overline{D_w(S)}\cap \overline{\mathbb{R}}=\varnothing$. Therefore $D_w(S)$  is bounded in $\mathbb{H}$, and we infer from Theorem~\ref{poq} that $S$ is a cocompact Fuchsian group. 
\end{proof}

\section{Schottky semigroups}\label{ooq}

We have now developed enough of the basic properties of semigroups to begin proving our main theorems. Before we do so, however, it is instructive to look at various examples of semidiscrete semigroups, to highlight some of the similarities and differences between discrete groups and semidiscrete semigroups. This task will occupy us for the next few sections.

Let us begin with a useful characterisation of finitely-generated semigroups that are semidiscrete and inverse free.

\begin{theorem}\label{kfa}
Suppose that $S$ is a semigroup generated by a finite collection $\mathcal{F}$ of M\"obius transformations. Then $S$ is semidiscrete and inverse free if and only if there is a nontrivial closed subset $X$ of $\overline{\mathbb{H}}$ that is mapped strictly inside itself by each member of $\mathcal{F}$.
\end{theorem}
\begin{proof}
Suppose first that there is a nontrivial closed subset $X$ of $\overline{\mathbb{H}}$ such that $f(X)\subset X$ and $f(X)\neq X$, for $f\in\mathcal{F}$. We define $z_f$ to be  a point in $\overline{\mathbb{H}}\setminus X$ such that $f(z_f)\in X$, for each $f\in\mathcal{F}$. Let $Y=\{z_f\,:\,f\in\mathcal{F}\}$, a finite set. Now, given any element $g$ of $S$ we can see, by writing $g$ as a word in $\mathcal{F}$, that $g(Y)\cap X\neq\varnothing$. It follows that the identity map $I$ is not contained in $\overline{S}$, so $S$ is semidiscrete and inverse free.

Conversely, suppose that $S$ is semidiscrete and inverse free. Choose a point $w$ in $\mathbb{H}$, and define $X$ to be the nontrivial closed set $\{w\}\cup \overline{S(w)}$, where, as usual, the closure is taken within $\overline{\mathbb{C}}$ (so $X$ is a subset of $\overline{\mathbb{H}}$). Since $S$ is semidiscrete and inverse free, the action of $S$ on $\mathbb{H}$ is strongly discontinuous, so there is a neighbourhood $U$ of $w$ such that $g(U)\cap U=\varnothing$, for every element $g$ of $S$. It follows that $w\notin f(X)$, for  $f\in\mathcal{F}$, so each member of $\mathcal{F}$ maps $X$ strictly inside itself. 
\end{proof}

Next we wish to introduce a large class of finitely-generated inverse-free semidiscrete semigroups that are somewhat related to Schottky groups. We should first specify exactly what we consider to be Schottky groups, as our definition differs slightly from others (and the definition in the introduction was for Schottky groups of rank two only).

Let $A_k$ and $B_k$, for $k=1,\dots,n$, be two collections of disjoint open intervals in $\overline{\mathbb{R}}$. For each index $k$, let $f_k$ be a M\"obius transformation that maps the complement of $\overline{A_k}$ to $B_k$. Each map $f_k$ is either parabolic or hyperbolic. A \emph{Schottky group} $G$ is a group generated as a group by such a collection of transformations. The \emph{rank} of the Schottky group is the integer $n$, and the collection $\{f_1,\dots,f_n\}$ is called a \emph{standard group-generating set} for $G$ (not every set of $n$ maps that generate $G$ as a group have the same sort of interval-mapping properties as the maps $f_k$). Schottky groups, as we have described them, are sometimes called \emph{classical Schottky groups}, and in some sources the \emph{closures} of the intervals are required to be pairwise disjoint, which prevents any of the generators from being parabolic. Schottky groups are discrete and free. 

Suppose now that $X$ is the union of a finite collection of disjoint closed intervals in $\overline{\mathbb{R}}$ (no singletons), and let $\mathcal{F}$ be a finite subset of $\mathcal{M}$ made up of transformations that map $X$ strictly within itself. A \emph{Schottky semigroup} is a semigroup generated by such a set $\mathcal{F}$. We use this terminology because Schottky groups contain many Schottky semigroups; for example, if $\{f,g\}$ is a standard group-generating set for a Schottky group of rank two, then $f$ and $g$ generate as a semigroup a Schottky semigroup.

We now give an example of a finitely-generated inverse-free semidiscrete semigroup that is not a Schottky semigroup. In this example, as usual, we denote the attracting and repelling fixed points of a hyperbolic element $g$ of $\mathcal{M}$ by $\alpha_g$ and $\beta_g$, respectively. Let $\{f,g,h\}$ be a set of hyperbolic maps that is a standard group-generating set for a rank three Schottky group. Choose these maps in such a way that $\alpha_f$ and $\alpha_h$ lie in different components of $\overline{\mathbb{R}}\setminus\{\alpha_g,\beta_g\}$. Let $k=fg^{-1}f^{-1}$, and define $S=\langle f,g,h,k\rangle$.  This semigroup lies in a discrete group, so it is semidiscrete. To see that $S$ is inverse free, suppose, in order to reach a contradiction, that $w_1\dotsb w_n=I$, where $w_i$ is a positive power of $f$, $g$, $h$, or $k$, for $i=1,\dots,n$. By thinking of $w_1\dotsb w_n$ as a word in $f$, $g$, and $h$, we see that $w_i$ cannot be a power of $f$ or $h$ (because the sum of the powers of each of $f$, $g$, and $h$ in this word must be 0, as those three maps generate as a group a free group). Therefore each map $w_i$ is equal to either $g^m$ or $fg^{-m}f^{-1}$, for some positive integer $m$, so clearly it is not possible for $w_1\dotsb w_n$ to equal the identity map after all.

It remains to prove that there is not a finite collection of disjoint, closed intervals in $\overline{\mathbb{R}}$ whose union $X$ is mapped strictly within itself by each element of $S$. Suppose there is such a set $X$. Then it must contain $\Lambda^+(S)$ (see Theorem~\ref{aiu} and the comments after that theorem), so in particular it contains $\alpha_g$. Furthermore, $X$ contains $g^n(\alpha_f)$ and $g^n(\alpha_h)$ for each positive integer $n$. These points accumulate on either side of $\alpha_g$ (our initial choice of $f$, $g$, and $h$ ensures that this is so). It follows that $\alpha_g$ is an interior point of $X$. Now, $k=fg^{-1}f^{-1}$, so $\beta_k =f(\alpha_g)$, which implies that $\beta_k$ is also an interior point of $X$. However, this is impossible because $k(X)\subset X$. Therefore $S$ is not a Schottky semigroup.

An important feature of this example is that it shows that the generating set for a finitely-generated inverse-free semidiscrete semigroup need not be unique, because $S=\langle fg,g,h,k\rangle$, as one can easily verify (in fact, one can remove the map $h$ from both generating sets to give a simpler example still).

\section{Generating new semigroups from old}

This section presents a simple combination theorem for generating semidiscrete semigroups. It is similar to combination theorems for discrete groups of Klein and Maskit \cite[Chapter~VII]{Ma1988}.

Given distinct points $u$ and $v$ in $\mathbb{H}$, recall that $K(u,v)=\{z\in\mathbb{H}\,:\, \rho(z,u)\leq\rho(z,v)\}$. Notice that $g(K(u,v))=K(g(u),g(v))$, for any M\"obius transformation $g$. If $S$ is a semidiscrete semigroup, and $w$ is a point of $\mathbb{H}$ that is not fixed by any nonidentity element of $S$, then we can express the Dirichlet region for $S$ centred at $w$ as 
\[
D_w(S)=\bigcap_{g\in S\setminus\{I\}} K(w,g(w)).
\]
Then
\[
\mathbb{H}\setminus D_w(S)=\bigcup_{g\in S\setminus\{I\}} K(g(w),w)^\circ,
\]
where $K(g(w),w)^\circ=\{z\in\mathbb{H}\,:\,\rho(z,g(w))<\rho(z,w)\}$, the interior of $K(g(w),w)$.

\begin{theorem}\label{ols}
Suppose that $S_1$ and $S_2$ are semidiscrete semigroups and $w$ is a point in $\mathbb{H}$ that is not fixed by any nonidentity element of $S_1$ or $S_2$. Suppose that $\mathbb{H}\setminus D_w(S_2^{-1})\subset D_w(S_1)$ and $\mathbb{H}\setminus D_w(S_1^{-1})\subset D_w(S_2)$. Then the semigroup $T$ generated by $S_1\cup S_2$ is semidiscrete. Furthermore, if $S_1$ and $S_2$ are inverse free, then so is~$T$. 
\end{theorem}
\begin{proof}
Observe that, for $g\in S_1\setminus\{I\}$ and $h\in S_2\setminus\{I\}$, we have 
\[
K(g(w),w)^\circ \subset \mathbb{H}\setminus D_w(S_1) \subset D_w(S_2^{-1})\subset K(w,h^{-1}(w)), 
\]
which implies that $K(g(w),w)^\circ\subset K(w,h^{-1}(w))^\circ$. Hence $h(K(g(w),w)^\circ)\subset K(h(w),w)^\circ$, and similarly $g(K(h(w),w)^\circ)\subset K(g(w),w)^\circ$. It follows that $h$ maps $\mathbb{H}\setminus D_w(S_1)$ into $\mathbb{H}\setminus D_w(S_2)$, and $g$ maps $\mathbb{H}\setminus D_w(S_2)$ into $\mathbb{H}\setminus D_w(S_1)$. 

Next, observe that $g(w)\in\mathbb{H}\setminus D_w(S_1)$ and $h(w)\in\mathbb{H}\setminus D_w(S_2)$. Choose $f\in T\setminus\{I\}$. By writing $f$ as a word with letters alternating between $S_1\setminus\{I\}$ and $S_2\setminus\{I\}$, we see that $f(w)$ belongs to either  $\mathbb{H}\setminus D_w(S_1)$ or $\mathbb{H}\setminus D_w(S_2)$. However, $S_1$ and $S_2$ are semidiscrete, so $w$ is contained in the interior of both $D_w(S_1)$ and $D_w(S_2)$. It follows that $T$ is semidiscrete.

Furthermore, if $S_1$ and $S_2$ are inverse free, then the same argument shows that the identity map does not belong to $T$. Hence $T$ is inverse free also.
\end{proof}

Theorem~\ref{ols} can be used as a tool for modifying known semidiscrete semigroups to generate new semidiscrete semigroups, as the following example illustrates. Let $U$ and $V$ be two closed hyperbolic half-planes that are separated by a large hyperbolic distance. Let $\mathcal{F}$ be a finite collection of M\"obius transformations that map the complement of $U$ into $V$, and let $S$ be the semigroup generated by $\mathcal{F}$. Choose $w$ to be the midpoint of the hyperbolic line segment between $U$ and $V$ that is orthogonal to the boundaries of $U$ and $V$ (it is not necessary to choose precisely this point, but this choice is a convenient one). Providing $U$ and $V$ are chosen to be sufficiently distant from one another (distance greater than $\log (3+2\sqrt{2}))$, we can find a closed interval $J$ in $\overline{\mathbb{R}}$ with the property that if $x\in J$, then the horodisc $H_x(w)$ intersects neither $U$ nor $V$. As a consequence, the $S$-orbit and $S^{-1}$-orbit of $w$ do not meet $H_x(w)$, so $x\in \overline{D_w(S)}\cap \overline{D_w(S^{-1})}$, by Theorem~\ref{aaw}. It follows that the closed hyperbolic half-plane $W$ with ideal boundary $J$ is contained in $D_w(S)\cap D_w(S^{-1})$.

Next, choose any Fuchsian group $G$ for which $D_w(G)$ contains a hyperbolic half-plane $K$. After conjugating $G$, we can assume that $\mathbb{H}\setminus K\subset W$. Then
\[
\mathbb{H}\setminus D_w(G) \subset \mathbb{H}\setminus K\subset W\subset D_w(S)\quad\text{and}\quad \mathbb{H}\setminus D_w(S^{-1}) \subset \mathbb{H}\setminus W\subset K\subset D_w(G).
\]
The theorem now tells us that the semigroup generated by $S\cup G$ is semidiscrete.

\section{Exceptional semigroups}
 
We recall that $\mathcal{M}(J)$ denotes the semigroup of M\"obius transformations that map a closed interval $J$ inside itself (recall also our convention that a closed interval is not a singleton or $\overline{\mathbb{R}}$). In this section we study the finitely-generated semidiscrete semigroups that lie within $\mathcal{M}(J)$. We include in this study a discussion of exceptional semigroups, which are not semidiscrete. 

It is helpful to define $\mathcal{M}_0(J)$ to be the group part of $\mathcal{M}(J)$, which comprises those M\"obius transformations that fix $J$ as a set.  This group is the one-parameter family of hyperbolic M\"obius transformations whose fixed points are the end points of $J$. If $J=[0,+\infty]$, then $\mathcal{M}_0(J)$ consists of all maps of the form $a z$, where $a>0$. To understand the finitely-generated semigroups in $\mathcal{M}(J)$, we first need to understand the finitely-generated semigroups in $\mathcal{M}_0(J)$. This is achieved in the following pair of results.
 
First, a lemma, which classifies the finitely-generated semigroups that are contained in the group $\{z+a\,:\,a\in\mathbb{R}\}$. Really this is the same as classifying additive semigroups of real numbers, so this lemma no doubt appears elsewhere in the literature in some form already.

\begin{lemma}\label{zps}
Let $S$ be a semigroup generated by maps $z+b_i$, where $b_i\in\mathbb{R}$, for $i=1,\dots,n$. Then exactly one of the following statements is true. 
\begin{enumerate}
\item\label{eea} Either $b_i \geq 0$ for $i=1,\dots,n$ or $b_i\leq 0$ for $i=1,\dots,n$, in which case the semigroup $S\setminus\{I\}$ is semidiscrete and inverse free.
\item\label{eeb} For every pair of indices $i$ and $j$, we have either $b_j=0$ or else $b_i/b_j \in \mathbb{Q}$ and one of the quotients $b_i/b_j$ is negative. In this case, $S$ is a discrete group.
\item\label{eec} Otherwise, there are two maps $z+b_i$ and $z+b_j$ that generate a dense subset of $\{z+b\,:\,b\in\mathbb{R}\}$.\qedhere
\end{enumerate}
\end{lemma}
\begin{proof}
In case {\ref{eea}}, it is clear that $S\setminus\{I\}$ is semidiscrete and inverse free.

In case {\ref{eeb}}, choose any of the numbers $b_i$ (other than $0$), and let $b_j$ be of opposite sign to $b_i$. Then $mb_i=-nb_j$ for some coprime positive integers $m$ and $n$. Choose positive integers $u$ and $v$ such that $vn-um=\pm1$. Let $\alpha=b_i/n=-b_j/m$. Then $vb_i+ub_j=\pm\alpha$. It follows that $S$ contains one of the maps $z+\alpha$ or $z-\alpha$. But $S$ also contains $z+n\alpha$ and $z-m\alpha$. So $S$ must contains \emph{both} $z+\alpha$ and $z-\alpha$, and in particular it must contain $z-b_i$. Therefore $S$ is a group, and one can use a short, standard argument from the theory of discrete groups to prove that $S$ is discrete.

In case {\ref{eec}}, we can choose a pair of nonzero numbers $b_i$ and $b_j$ such that $b_i/b_j\notin \mathbb{Q}$. We can assume that $b_i$ and $b_j$ have opposite signs: if at first they do not, then replace one of them by another number $b_k$ of the opposite sign (if you are careful about which of $b_i$ and $b_j$ you replace, then you can be sure that the quotient of the two numbers you end up with is irrational). Now let $(u_n/v_n)$ be a sequence of rational numbers, where $u_n,v_n\in\mathbb{N}$, that satisfies 
\[
\left| \frac{b_i}{b_j}+\frac{u_n}{v_n}\right| < \frac{1}{2v_n^2}.
\]
Then $|v_nb_i+u_nb_j|\to 0$ as $n\to\infty$, and we know that $v_nb_i+u_nb_j\neq 0$. From this we see that the maps $z+v_nb_i+u_nb_j$ accumulate at the identity map, so $\langle z+b_i,z+b_j\rangle$ is not semidiscrete.
\end{proof}

Next we state a comparable result to Lemma~\ref{zps} for maps of the form $az$, where $a>0$. It is proved by observing that $\log x$ is an isometry from the positive real axis equipped with the one-dimensional hyperbolic metric to $(\mathbb{R},+)$.

\begin{corollary}\label{nps}
Let $S$ be a semigroup generated by maps $a_iz$, where $a_i>0$, for $i=1,\dots,n$. Then exactly one of the following statements is true. 
\begin{enumerate}
\item\label{wea} Either $a_i \geq 1$ for $i=1,\dots,n$ or $a_i\leq 1$ for $i=1,\dots,n$, in which case the semigroup $S\setminus\{I\}$ is semidiscrete and inverse free.
\item\label{web} For every pair of indices $i$ and $j$, we have either $a_j=1$ or else $\log a_i/\log a_j \in \mathbb{Q}$ and one of the quotients $\log a_i/\log a_j$ is negative. In this case, $S$ is a discrete group.
\item\label{wec} Otherwise, there are two maps $a_iz$ and $a_jz$ that generate a dense subset of $\{az\,:\,a>0\}$.\qedhere
\end{enumerate}
\end{corollary}
 
Suppose now that $S$ is a finitely-generated semidiscrete semigroup contained in $\mathcal{M}(J)$. Then $S\cap \mathcal{M}_0(J)$ is either empty or else falls into one of the classes {\ref{wea}} or {\ref{web}} from Corollary~\ref{nps}. We examine these two possibilities, in turn.

\begin{theorem}\label{jjr}
Suppose that $S$ is a finitely-generated semigroup that lies in $\mathcal{M}(J)$, for some closed interval $J$, such that the elements of $S\setminus\{I\}$ that fix $J$ as a set form a semigroup that is semidiscrete and inverse free. Then $S\setminus\{I\}$ is also a semigroup that is semidiscrete and inverse free.
\end{theorem}
\begin{proof}
By conjugation, we can assume that $J=[0,+\infty]$. Let $\mathcal{F}$ be a finite generating set for $S$. Let $a_iz$, where $a_i>0$, for $i=1,\dots,n$, be the members of $\mathcal{F}$ that fix $J$ as a set. By Corollary~\ref{nps}, $a_i \geq 1$ for all $i$, or $a_i\leq 1$ for all $i$; let us assume the latter, as the other case is similar (and in fact one can move from one case to the other by conjugating $S$ by the map $1/z$).

We will now construct a nontrivial closed subset $X$ of $\overline{\mathbb{H}}$ that is mapped strictly inside itself by each member of $\mathcal{F}\setminus\{I\}$. Theorem~\ref{kfa} can then be applied to see that $S\setminus\{I\}$ is semidiscrete and inverse free.

Let $D$ denote the closed top-right quadrant of the extended complex plane. This is the closure in $\overline{\mathbb{H}}$ of the hyperbolic half-plane with boundary $[0,+\infty]$. If $f\in\mathcal{F}$ and $f(\infty)\neq \infty$, then $f(D)$ is the closure in $\overline{\mathbb{H}}$ of the hyperbolic half-plane with boundary $[f(0),f(\infty)]$. Since $\mathcal{F}$ is finite, we can choose $v>0$ such that $f(D)\subset \{z\in D\,: \, \Im z<v\}$ for all such maps $f$.

Suppose now that $f$ is a hyperbolic element of $\mathcal{F}$ with $\alpha_f=\infty$; then $\beta_f<0$. We choose $u<0$ such that $\beta_f<u$ for all such hyperbolic maps $f$.

Let $\ell$  be the Euclidean line through $u$ and $iv$. We define $X$ to be the closed subset of $\overline{\mathbb{H}}$ composed of those points of $D$ that lie on or below the line $\ell$; see Figure~\ref{kbr}.

\begin{figure}[ht]
	\centering
	\begin{tikzpicture}
	\fill[fill=col1] (0,0) -- (0,2) -- (1,4) -- (4,4) -- (4,0) -- cycle;
		\draw (-4,0) -- (4,0);
	\draw (0,0) -- (0,4);
	\draw (-1,0) -- (1,4);
	\node[label=below:$u$,draw=none,fill=black,minimum size=3pt] at (-1,0){};
	\node[label=left:$iv$,draw=none,fill=black,minimum size=3pt] at (0,2){};
	\node[draw=none,fill=none] at (2,2){$X$};
	\node[draw=none,fill=none] at (-2,3){$\mathbb{H}$};
	\node[draw=none,fill=none] at (0.5,3.4){$\ell$};
	\end{tikzpicture}
\caption{The closed subset $X$ of $\overline{\mathbb{H}}$}
\label{kbr}
\end{figure}

It remains to show that each member $f$ of $\mathcal{F}\setminus\{I\}$ maps $X$ strictly inside itself. This is certainly true if $f(\infty)\neq \infty$ because
\[
f(X)\subset f(D) \subset\{z\in D\,: \, \Im z<v\} \subset X.
\]
Next, if $f$ is a hyperbolic map with $\alpha_f=\infty$, then $f(z)=\lambda(z-\beta_f)+\beta_f$, where $\lambda>1$ and $\beta_f<u$, and again we can see that $f(X)$ is strictly contained in $X$. The remaining two possibilities are that 
 $f$ is parabolic with fixed point $\infty$ (of the form $z+t$, for $t>0$) or $f$ is hyperbolic with $\beta_f=\infty$, in which case $f(z)=\lambda(z-\alpha_f)+\alpha_f$, where $\lambda<1$ and $\alpha_f\geq 0$ (which includes the possibility that $f$ is one of the maps $a_iz$). In both cases $f$ maps $X$ strictly inside itself. 
\end{proof}

Theorem~\ref{jjr} concludes that $S\setminus\{I\}$ is a semigroup that is semidiscrete and inverse free. As a consequence, $S$ is semidiscrete, but of course it is not inverse free if it happens to contain the identity map.

The next theorem considers those finitely-generated semidiscrete semigroups in $\mathcal{M}(J)$ whose group parts are nontrivial discrete subgroups of $\mathcal{M}_0(J)$. Before the theorem, we need a lemma about generating sets.

\begin{lemma}\label{zko}
Suppose that $S$ is a semigroup contained in $\mathcal{M}(J)$, for some closed interval $J$, that is generated by a set $\mathcal{F}$. Then  $S$ has an element that fixes exactly one of the end points of $J$ if and only if $\mathcal{F}$ contains such an element.
\end{lemma} 
\begin{proof}
If $\mathcal{F}$ has an element that fixes exactly one of the end points of $J$, then certainly $S$ does too, as $\mathcal{F}\subset S$. Conversely, suppose that no element of $\mathcal{F}$ fixes exactly one of the end points of $J$. Let $f$ be a map in $S$ that fixes one of the end points $a$ of $J$; we will show that in fact $f(J)=J$, so $f$ fixes both end points of $J$. 

Choose maps $f_i$ in $\mathcal{F}$ such that $f=f_1\dotsb f_n$.  Observe that $f(J)\subset f_1(J)$, so $a\in f_1(J)$, which implies that $f_1(a)=a$. Therefore the element $f_2\dotsb f_n$ of $S$ fixes $a$ as well, and, by our assumption, $f_1(J)=J$. Repeating this argument we see that $f_i(J)=J$, for $i=1,\dots,n$, so $f(J)=J$, as required.
\end{proof}

\begin{theorem}\label{ioe}
Suppose that $S$ is a finitely-generated semigroup that lies in $\mathcal{M}(J)$, for some closed interval $J$, such that the elements of $S$ that fix $J$ as a set form a nontrivial discrete group. Then $S$ is semidiscrete if and only if no other element of $S$ outside this discrete group fixes either of the end points of $J$.
\end{theorem}
\begin{proof}
By conjugation, we can assume that $J=[0,+\infty]$. Let $\mathcal{F}$ be a finite generating set for $S$. Let $\mathcal{F}_0$ denote those elements of $\mathcal{F}$ that fix $[0,+\infty]$ as a set. We are assuming that $\mathcal{F}_0$ generates a semigroup $\langle \lambda z, z/\lambda \rangle$, for some constant $\lambda>1$. By Lemma~\ref{zko}, it suffices to prove that $S$ is semidiscrete if and only if no element of $\mathcal{F}\setminus\mathcal{F}_0$ fixes $0$ or $\infty$.

Suppose, then, that no element of $\mathcal{F}\setminus\mathcal{F}_0$ fixes $0$ or $\infty$. Then there is an interval $K=[s,t]$, where $0<s<t<+\infty$, such that if $f\in\mathcal{F}\setminus\mathcal{F}_0$, then $f(J)\subset K$. Suppose now that $(F_n)$ is a sequence of distinct elements of $S$. We wish to show that $(F_n)$ cannot converge uniformly to the identity map, because doing so will demonstrate that $S$ is semidiscrete. If $F_n\in \langle \lambda z, z/\lambda \rangle$ for infinitely many integers $n$, then certainly $(F_n)$ cannot converge uniformly to the identity map because $\langle \lambda z, z/\lambda \rangle$ is discrete. Otherwise, when $n$ is sufficiently large, we can write $F_n=G_nf_nH_n$, where $G_n \in \langle \lambda z, z/\lambda\rangle$, $f_n\in\mathcal{F}\setminus\mathcal{F}_0$, and $H_n\in S$. Notice that
\[
f_nH_n(J) \subset f_n(J)\subset K.
\]
Let $G_n(z)=\lambda^{k_n}z$, where $k_n\in\mathbb{Z}$. Either $k_n\geq 0$, in which case $F_n(0) \geq f_nH_n(0)\geq s$, or $k_n< 0$, in which case $0<F_n(\infty)< f_nH_n(\infty)\leq t$. So $(F_n)$ does not converge uniformly to the identity map.

Suppose now that there is an element $f$ of $\mathcal{F}\setminus\mathcal{F}_0$ that fixes $0$ or $\infty$. Let us assume that $f$ fixes $\infty$ (the other case can be deduced from this case by conjugating $S$ by the map $1/z$). Then $f(z)=az+b$, where $a,b>0$. Let $g(z)=\lambda z$. Observe that
\[
g^{-n} f g^{n}(z) = az + \lambda^{-n}b.
\]
Therefore the map $h(z)=az$ is an accumulation point of $S$ in $\mathcal{M}$. By composing powers of $h$ with the maps $g$ and $g^{-1}$ we see that the identity map is an accumulation point of $S$ too, so $S$ is not semidiscrete.
\end{proof}
 
We recall from the introduction that an exceptional semigroup $S$ is a semigroup that lies in $\mathcal{M}(J)$, for some closed interval $J$, such that the elements of $S$ that fix $J$ as a set form a nontrivial discrete group, and such that there is another element of $S$ outside this discrete group that fixes one of the end points of $J$. Theorem~\ref{ioe} tells us that exceptional semigroups are not semidiscrete.

Let us finish here by explaining how to determine whether a finitely-generated semigroup $S$ is exceptional by examining any generating set $\mathcal{F}$ of $S$ alone. First look for a collection of two or more hyperbolic transformations in $\mathcal{F}$ with the same axes that together generate a discrete group. Corollary~\ref{nps} tells us how to test whether a set of hyperbolic transformations with the same axes generate such a group. If there is no collection of this type, or more than one collection, then $S$ is not exceptional. Let us suppose that there is indeed one such collection, with axis $\gamma$, and this collection does generate a discrete group. In order for $S$ to be exceptional, every element of $\mathcal{F}$ must lie in $\mathcal{M}(J)$, where $J$ is one of the two intervals in $\overline{\mathbb{R}}$ with the same end points as $\gamma$. If this is the case, then Lemma~\ref{zko} tells us that $S$ is exceptional if and only if $\mathcal{F}$ contains an element that fixes exactly one of the end points of $J$.

\section{Elementary semigroups}\label{oiw}

A group of M\"obius transformations is said to be elementary if it has a finite orbit in $\overline{\mathbb{H}}$. Accordingly, we define a semigroup $S$ to be elementary if $S$ has a finite orbit in   $\overline{\mathbb{H}}$. This definition is in conflict with that of \cite{FrMaSt2012}, where a semigroup is considered to be elementary if $|\Lambda^-(S)|\leq 2$. This alternative definition is weaker than the definition we use because the condition $|\Lambda^-(S)|\leq 2$ implies that $S$ has a finite orbit in $\overline{\mathbb{H}}$ (see \cite[Theorem 2.11]{FrMaSt2012}).

The alternative definition of an elementary semigroup has two undesirable consequences for our purposes: first, with the alternative definition,  it is possible for $S$ to be elementary and $S^{-1}$ not elementary; second, with the alternative definition and when $S$ happens to be a group, it is possible for $S$ to be an elementary group but not an elementary semigroup. However, the definition of \cite{FrMaSt2012} is more suitable for the purposes of that paper, in which semigroups of rational maps are considered.

The elementary semigroups are classified in the following theorem.

\begin{theorem}\label{ndo}
Let $S$ be an elementary semigroup. Then one of the following holds: 
\begin{enumerate}
\item\label{pja} there is a point in $\mathbb{H}$ fixed by all elements of $S$
\item\label{pjb} there is a point in $\overline{\mathbb{R}}$ fixed by all elements of $S$, or
\item\label{pjc} there is a pair of points in $\overline{\mathbb{R}}$ fixed as a set by all elements of $S$.\qedhere
\end{enumerate}
\end{theorem}

In case~{\ref{pja}} we can, after conjugating, obtain a semigroup with common fixed point $i$. In that case, all maps in $S$ are of the form $(az-b)/(bz+a)$, where $a^2+b^2=1$. In case~{\ref{pjb}}, we can conjugate so that the fixed point is $\infty$. Then all maps in $S$ are of the form $az+b$, where $a>0$ and $b\in\mathbb{R}$. In case~{\ref{pjc}}, we can conjugate so that the pair of fixed points is $\{0,\infty\}$. Then each map of $S$ has the form $az$ or $a/z$, where $a>0$.

\begin{proof}[P\textbf{roof of Theorem~\ref{ndo}}]
Since $S$ is elementary, it has a finite orbit, $X$ say. If $g\in S$ and $x\in X$, then the sequence $x,g(x),g^2(x),\dotsc$, which lies in $X$, must have two terms that are equal. It follows that there is a positive integer $m_x$ for which $g^{m_x}(x)=x$. Let $m$ denote the product of all the integers $m_x$, for $x\in X$. Then $g^m(x)=x$ for any element $x$ of $X$. If $X$ has three or more elements -- or if it has two elements, at least one of which lies in $\mathbb{H}$ -- then $g^m$ must be the identity map $I$. It follows that $g$ itself is either elliptic of finite order or else equal to $I$. Therefore $S$ is a group. It is well known, and easy to show, that all elements of such a group have a common fixed point in $\mathbb{H}$, which is case~{\ref{pja}}.

The remaining possibilities are that either $X$ consists of a single point in $\overline{\mathbb{R}}$, which is case~{\ref{pjb}}, or $X$ consists of a pair of points in $\overline{\mathbb{R}}$, which is case~{\ref{pjc}}.
\end{proof}

The next observation follows immediately from the theorem.

\begin{corollary}
The semigroup $S$ is elementary if and only if $S^{-1}$ is elementary.
\end{corollary}

The elementary semigroups that are of most interest to us  are those that are finitely generated and semidiscrete, which are classified in the theorem that follows.  To appreciate this result, it helps to observe that $b/(1-a)$ is the finite fixed point of the map $az+b$, where $a\neq 1$. This fixed point is attracting if $a<1$, and repelling if $a>1$.

\begin{theorem}\label{mme}
Let $S$ be a finitely-generated elementary  semidiscrete semigroup. Then one of $S$ or $S^{-1}$ is conjugate in $\mathcal{M}$ to a semigroup of one of the following types:
\begin{enumerate}
\item\label{cla} a finite cyclic group generated by an elliptic transformation
\item\label{clb} one of the groups $\langle az,z/a\rangle$ or $\langle az,z/a,-1/z\rangle$, where $a>1$
\item\label{clg} $\langle a_1z+b_1,\dots,a_mz+b_m, z+1,z-1\rangle$, where $a_i>1$ and $b_i\in\mathbb{R}$ for all $i$
\item\label{cle} $\langle a_iz+b_i,c_jz+d_j,z+t_k\,:\,i=1,\dots,m,j=1,\dots,n,k=1,\dots,r\rangle$, where $a_i>1$, $0<c_j<1$, $b_i,d_j\in\mathbb{R}$, $t_k\geq 0$, and $b_i/(1-a_i)<d_j/(1-c_j)$ for all $i,j,k$. 
\end{enumerate}
In classes~\ref{clg} and \ref{cle}, $m$, $n$, and $r$ are nonnegative integers, not all zero in class~\ref{cle}.
\end{theorem}

Semigroups of types~{\ref{cla}} and~{\ref{clb}} are discrete groups. If $S$ is a semigroup of type~{\ref{cle}} (other than the trivial semigroup, $\{I\}$), then $S\setminus\{I\}$ is a Schottky semigroup (which implies that $S$ is semidiscrete) because all elements of $S\setminus\{I\}$ map the closed interval $[s,+\infty]$ strictly inside itself, where $\max_i b_i/(1-a_i) < s < \min_j d_j/(1-c_j)$.  Semigroups of type~{\ref{clg}} have nonempty inverse-free part (unless $m=0$) and nonempty group part. One can easily check that such semigroups are semidiscrete.

We need the following lemma and corollary to prove Theorem~\ref{mme}.

\begin{lemma}\label{kqw}
Consider the semigroup $S$ generated by $f(z)=a z+b$ and $g(z)=cz+d$, where $a>1$ and $0<c<1$, and $b/(1-a)<d/(1-c)$. The closure $\overline{S}$ of this semigroup in $\mathcal{M}$ contains all maps of the form $z+t$, where $t\geq d/(1-c) - b/(1-a)$.
\end{lemma}
\begin{proof}
Observe that 
\[
g^nf^m(z)=a^mc^nz+bc^n\left(\frac{1-a^m}{1-a}\right)+d\left(\frac{1-c^n}{1-c}\right)=a^mc^nz+t_0(1-c^n)+\frac{b}{1-a}(1-a^mc^n),
\]
where $m,n\in\mathbb{N}$ and $t_0=d/(1-c) - b/(1-a)$. Given $\varepsilon>0$, we can choose $m$ and $n$ such that $c^n<\varepsilon$ and $|a^mc^n-1|<\varepsilon$. Therefore $h(z)=z+t_0\in\overline{S}$. Next, observe that
\[
g^nh^kf^m(z) =a^mc^nz+t_0(1-c^n)+\frac{b}{1-a}(1-a^mc^n) +kt_0 c^n,
\]
where $m,k,n\in\mathbb{N}$. Let $s>0$. Given $\varepsilon>0$, with $\varepsilon<s$, we can choose $m$ and $n$ such that $c^n<\varepsilon$, $t_0c^n<\varepsilon$, and $|a^mc^n-1|<\varepsilon$. We can then choose $k$ such that $|kt_0c^n-s|<\varepsilon$. Therefore $z+t_0+s\in\overline{S}$, and the result follows.
\end{proof}

\begin{corollary}\label{oyy}
Consider the semigroup $S$ generated by $f_i(z)=a_iz+b_i$, for $i=1,2,3$, where $a_1,a_3>1$, $0<a_2<1$, $b_1,b_2,b_3\in\mathbb{R}$, and
\[
\frac{b_1}{1-a_1}\leq \frac{b_2}{1-a_2}\leq \frac{b_3}{1-a_3},
\]
with equality in at most \textbf{one} of these two inequalities. Then $S$ is not semidiscrete.
\end{corollary}
\begin{proof}
Let $p_i=b_i/(1-a_i)$, for $i=1,2,3$. Suppose that $p_1<p_2<p_3$. By applying Lemma~\ref{kqw} to the semigroups $\langle f_1,f_2\rangle$ and $\langle f_2^{-1},f_3^{-1}\rangle$, we see that $\overline{S}$ contains all maps of the form $z+t$, where $t\geq p_2-p_1$ or $t\leq p_2-p_3$. Therefore $\overline{S}$  contains all maps of the form $z+t$, where $t\in\mathbb{R}$, so $S$ is not semidiscrete.

Suppose now that $p_1=p_2<p_3$ (the other case with $p_1<p_2=p_3$ is similar). By Corollary~\ref{nps}, if the semigroup generated by $f_1$ and $f_2$ is semidiscrete, then it must in fact be a nontrivial discrete group. But then $S$ is an exceptional semigroup, so it is not semidiscrete, by Theorem~\ref{ioe}.
\end{proof}

Let us now prove Theorem~\ref{mme}.

\begin{proof}[P\textbf{roof of Theorem~\ref{mme}}]
If $S$ is trivial (equal to $\{I\}$), then it is accounted for in class~{\ref{cle}}, so let us assume henceforth that $S$ is not trivial.

Following Theorem~\ref{ndo}{\ref{pja}}, let us suppose first that there is a point $p$ in $\mathbb{H}$ fixed by all elements of $S$. Then each nonidentity map in $S$ is an elliptic rotation about $p$. If one of these maps has infinite order, then $S$ is not semidiscrete. Otherwise, each map has finite order, so $S$ is a group, and it must be a cyclic group, which falls into class~{\ref{cla}}.

Suppose now that we have the circumstances of Theorem~\ref{ndo}{\ref{pjc}}, so there is a pair of points in $\overline{\mathbb{R}}$ that are fixed as a set by each element of $S$. As we have observed already, we can, after conjugating $S$, assume that these points are $\{0,\infty\}$, in which case each generator of $S$ has the form $az$ or $-a/z$, where $a>0$. Corollary~\ref{nps} tells us that if there are no generators of the latter type, then one of $S$ or $S^{-1}$ is either of class~{\ref{clb}}, or else of class~{\ref{cle}} (with $m>0$, $n=0$, all $b_i=0$, and either $r=0$, or $r=1$ and $t_1=0$).

Suppose instead that there is a generator of $S$ of the form $-a/z$. After conjugating we can assume that the map is $h(z)=-1/z$. Observe that if $f(z)=az$, then $hfh(z)=f^{-1}(z)$. It follows that $S$ is a discrete group, so it must be of class~{\ref{clb}}.

It remains to consider the situation of Theorem~\ref{ndo}{\ref{pjb}}, in which there is a point in $\overline{\mathbb{R}}$ fixed by all elements of $S$. We assume by conjugation that this point is $\infty$, so  each element of $S$ has the form $az+b$, where $a>0$ and $b\in\mathbb{R}$. 
Let $\{a_iz+b_i,c_jz+d_j,z+t_k\,:\,i=1,\dots,m,j=1,\dots,n,k=1,\dots,r\rangle$ be a generating set for $S$, where $a_i>1$, $0<c_j<1$,  and $b_i,d_j,t_k\in\mathbb{R}$ for all $i,j,k$, and $m$, $n$, and $r$ are nonnegative integers, not all zero. 

Let $S_0$ be the semigroup generated by the maps $z+t_j$, which is also semidiscrete. From Lemma~\ref{zps} we see that, after conjugating $S$, we can assume that either (a) $t_k\geq 0$ for all $k$, or $t_k\leq 0$ for all $k$, or (b) $S_0=\langle z+1,z-1\rangle$.

Suppose first that at least one of $m$ or $n$ is $0$. If $S_0$ is of type~(a), then one of $S$ or $S^{-1}$ is of class~{\ref{cle}}. If $S_0$ is of type~(b), then one of $S$ or $S^{-1}$ is of class~{\ref{clg}}.

We can assume now that $m$ and $n$ are positive. After replacing $S$ with $S^{-1}$ if need be, we can also assume that $\min_i b_i/(1-a_i)\leq \min_j d_j/(1-c_j)$.

Suppose in addition that $b_i/(1-a_i)<d_j/(1-c_j)$ for all $i,j$. Then, by Lemma~\ref{kqw}, $\overline{S}$ contains all maps of the form $z+t$, where $t$ is greater than some constant $t_0$. If $t_k<0$ for some $k$, then $\overline{S}$ contains all maps of the form $z+t$, $t\in\mathbb{R}$, so $S$ is not semidiscrete. We deduce that $t_k\geq 0$ for all $k$, so $S$ is of class~{\ref{cle}}.

Suppose finally that $d_j/(1-c_j)\leq b_i/(1-a_i)$ for some $i,j$. Corollary~\ref{oyy} tells us that, because $S$ is semidiscrete, we must have $b_i/(1-a_i)=d_j/(1-c_j)$ for \emph{all} $i,j$ -- all the hyperbolic maps have the same fixed points, which we can assume, after conjugating $S$ by a transformation that fixes $\infty$, are $0$ and $\infty$. The semigroup generated by these hyperbolic maps must be a nontrivial discrete group. Since $S$ is semidiscrete, and it contains a nontrivial discrete group of transformations that fixes each of $[-\infty,0]$ and $[0,+\infty]$ as a set, we see from Theorem~\ref{ioe} that $S$ cannot contain a parabolic map that fixes $\infty$. Hence $r=0$, and $S$ is of class~{\ref{clb}}.
\end{proof}

\section{Proof of Theorem~\ref{cof}}\label{mnm}

Here we prove our first main theorem, Theorem~\ref{cof}. We begin with a basic lemma on hyperbolic geometry. In this lemma we refer to a \emph{Stolz region} based at a point $p$ in $\overline{\mathbb{R}}$, which is a set of the form
\(
\{z\in \mathbb{H} : \rho(z,\gamma) < k\},
\)
where $\gamma$ is a hyperbolic line with one endpoint at $p$, and $k$ is a positive constant.

\begin{lemma}\label{zji}
Suppose that $S$ is a semigroup generated by a finite set $\mathcal{F}$ of M\"obius transformations, and suppose that $(F_n)$ is a composition sequence generated by $\mathcal{F}$ that is an escaping sequence. Let $L=\Lambda_c^+(F_n)$. Then $F_m^{-1}(L) \subset \Lambda_c^+(S)$, for any positive integer~$m$.
\end{lemma}
\begin{proof}
Let $p\in L$. Choose $w\in \mathbb{H}$. Then there is a sequence of positive integers $n_1, n_2,\dotsc$ such that the sequence $(F_{n_i}(w))$ converges to $p$ (in the chordal metric) within some Stolz region based at $p$. Therefore, for any positive integer $m$, the sequence $(F_m^{-1}F_{n_i}(w))$ converges to $F_m^{-1}(p)$ in some Stolz region based at $F_m^{-1}(p)$. Now, if $n_i>m$, then $F_m^{-1}F_{n_i}\in S$. Therefore $F_m^{-1}(p)\in \Lambda_c^+(S)$, as required.
\end{proof}

The lemma remains true if we replace the forward \emph{conical limit sets} with forward \emph{limit sets} (however, we do not need this alternative statement).

\begin{theorem}\label{cja}
Suppose that $S$ is a semigroup generated by a finite set $\mathcal{F}$ of M\"obius transformations. Suppose also that there is a composition sequence $(F_n)$ generated by $\mathcal{F}$ that is an escaping sequence that does not converge ideally. Then $\Lambda^+_c(S)=\overline{\mathbb{R}}$, unless   $\Lambda^-(S)=\{q\}$ for some point $q$, in which case $\Lambda^+_c(S)$ might equal $\overline{\mathbb{R}}\setminus\{q\}$ instead.\end{theorem}
\begin{proof}
Choose any point $w$ in $\mathbb{H}$. Let $k=\max\{\rho(w,f(w))\,:\,f\in\mathcal{F}\}$. We are given that the sequence $(F_n(w))$ accumulates, in the chordal metric, at two distinct points $a$ and $b$ in $\overline{\mathbb{R}}$. Suppose there is a point $d$ other than $a$ and $b$ that is \emph{not} a forward conical limit point of $(F_n)$. By conjugating $(F_n)$ if need be we can assume that $d=\infty$ and $a<b$. Let $\gamma_c$ be the hyperbolic geodesic with one end point $c$ inside $(a,b)$ and the other at $\infty$. Let $\Gamma_c=\{z\in\mathbb{H}\,:\, \rho(z,\gamma_c)< k\}$. Notice that 
\[
\rho(F_{n-1}(w),F_n(w)) = \rho(w,f_n(w))\leq k.
\]
It follows that infinitely many terms from the sequence $(F_n(w))$ lie in $\Gamma_c$. This infinite set inside $\Gamma_c$ cannot accumulate at $\infty$ because $\infty$ is not a forward conical limit point of $(F_n)$, so it must accumulate at~$c$. Hence~$c$ is a forward conical limit point of $(F_n)$.

We have shown that $(a,b)\subset \Lambda_c^+(F_n)$. Choose a point $y$ in $(a,b)$. Since $\Lambda_c^-(F_n^{-1})=\Lambda_c^+(F_n)$, we can apply Lemma~\ref{kka} to deduce the existence of a sequence of positive integers $n_1,n_2,\dotsc$ and two distinct points $p$ and $q$ in $\overline{\mathbb{R}}$ such that $F_{n_i}^{-1}(y)\to p$ and $F_{n_i}^{-1}(x)\to q$ for $x\in\overline{\mathbb{R}}\setminus\{y\}$. It follows that any point $u$ in $\overline{\mathbb{R}}\setminus\{q\}$ is contained in $F_{n_i}^{-1}\big((a,b)\big)$, provided $n_i$ is sufficiently large. Lemma~\ref{zji} tells us that $F_{n_i}^{-1}\big((a,b)\big) \subset \Lambda_c^+(S)$, so we see that $\overline{\mathbb{R}}\setminus\{q\}\subset \Lambda_c^+(S)$.

To finish, let us suppose that $\Lambda^-(S)\neq \{q\}$. There are now two cases to consider: either there is an element $f$ of $S$ that does not fix $q$, or there is no such element. In the first case, we have $q=f(v)$ for some point $v$ in $\overline{\mathbb{R}}\setminus\{q\}$. As $\Lambda_c^+(S)$ is forward invariant under $S$, we see that $q$ is an element of $\Lambda_c^+(S)$, so $\Lambda_c^+(S)=\overline{\mathbb{R}}$. In the second case, all elements of $S$ fix $q$, and since $\Lambda^-(S)\neq \{q\}$, there must be a hyperbolic element with attracting fixed point $q$. Hence $q\in\Lambda_c^+(S)$, so again $\Lambda_c^+(S)=\overline{\mathbb{R}}$.
\end{proof}

The exceptional case in which $\Lambda^-(S)=\{q\}$ and $\Lambda_c^+(S)=\overline{\mathbb{R}}\setminus\{q\}$ certainly can arise. For example, let $S=\left\langle \tfrac12 z, z+1,z-1\right\rangle$. Clearly $\Lambda^-(S)=\{\infty\}$, as all hyperbolic elements of $S$ have repelling fixed point $\infty$. One can easily construct a composition sequence generated by $\{\tfrac12 z, z+1,z-1\}$ that is an escaping sequence that does not converge ideally, and one can check that $\Lambda_c^+(S)$ contains $\mathbb{R}$. However, $\infty\notin \Lambda_{c}^+(S)$ because all elements of $S$ have the form $az+b$, where $a\leq 1$, so even though the orbit of a point in $\mathbb{H}$ under $S$ does accumulate at $\infty$ in the chordal metric, it does not do so within a Stolz region based at $\infty$.

 We can now prove one of the most significant results of this paper, from which we will deduce Theorem~\ref{cof}.

\begin{theorem}\label{kkl}
Let $S$ be a semidiscrete semigroup generated by a finite  collection $\mathcal{F}$ of M\"obius transformations. Suppose that $|\Lambda^-(S)|\neq 1$.  Then the following statements are equivalent:
\begin{enumerate}
\item every composition sequence generated by $\mathcal{F}$ that is an escaping sequence converges ideally
\item $S$ is not a cocompact Fuchsian group.\qedhere
\end{enumerate}
\end{theorem}
\begin{proof}
Suppose first that $S$ is a cocompact Fuchsian group. Then it has a finite-sided fundamental polygon $D$, compact in $\mathbb{H}$. The images of $D$ under $S$ tessellate $\mathbb{H}$. Let $\{g(D)\,:\, g\in\mathcal{G}\}$ be the polygons in this tessellation that are adjacent to $D$, where $\mathcal{G}$ is some finite subset of $S$. Next, we define $(D_n)$ to be a sequence of polygons in the tessellation such that $D_0=D$, and $D_{n-1}$ is adjacent to $D_n$, for each $n$. Let $G_n$ be the unique element of $S$ such that $G_n(D)=D_n$, including $G_0=I$. We also define $g_n=G_{n-1}^{-1}G_n$,  so $G_n=g_1\dotsb g_n$. Now, $D_{n-1}$ and $D_n$ are adjacent polygons in the tessellation, so $G_{n-1}^{-1}(D_{n-1})=D$ and $G_{n-1}^{-1}(D_n)=g_n(D)$ are also adjacent polygons. Therefore $g_n\in\mathcal{G}$.

Next we choose a point $z$ in the interior of $D$, and define $z_n=G_n(z)$. We can assume that our polygons $(D_n)$ are chosen such that $(z_n)$ accumulates in the chordal metric at two points in $\overline{\mathbb{R}}$, but does not accumulate in $\mathbb{H}$. It follows that $(G_n)$ is a composition sequence generated by $\mathcal{G}$ that is an escaping sequence that does not converge ideally. If we express each element of $\mathcal{G}$ as a finite composition of elements of $\mathcal{F}$, then we obtain a composition sequence generated by $\mathcal{F}$  that is an escaping sequence that does not converge ideally.

Suppose, conversely, that there is a composition sequence generated by $\mathcal{F}$ that is an escaping sequence that does not converge ideally. By Theorem~\ref{cja}, we see that  $\Lambda^+_c(S)=\overline{\mathbb{R}}$. Corollary~\ref{waa} then tells us that $S$ is a cocompact Fuchsian group.
\end{proof}

Theorem~\ref{kkl} fails if we remove the hypothesis $|\Lambda^-(S)|\neq 1$; we have seen an example of this already, just before the statement of the theorem.

\begin{corollary}\label{axo}
Let $S$ be a semidiscrete semigroup generated by a finite collection $\mathcal{F}$ of M\"obius transformations. Suppose that $S$ is not a cocompact Fuchsian group. If $F_n=f_1\dotsb f_n$ is a composition sequence generated by $\mathcal{F}$, then either $(F_n)$ converges ideally, or $f_n$ lies in the group part of $S$ for all sufficiently large positive integers~$n$.
\end{corollary}
\begin{proof}
If $(F_n)$ does not converge ideally, then, since $S$ is not a cocompact Fuchsian group, $(F_n)$ is not an escaping sequence, by Theorem~\ref{kkl}. We deduce from Lemma~\ref{cii} that there is a subsequence $(F_{n_k})$, where $n_1<n_2<\dotsb$, that converges uniformly to a M\"obius transformation $F$. Hence $F_{n_{k}}^{-1}F_{n_{k+1}}\to I$ as $k\to\infty$. Since $F_{n_{k}}^{-1}F_{n_{k+1}}=f_{n_{k}+1}\dotsb f_{n_{k+1}}\in S$, and $S$ is semidiscrete, we see that $f_{n_{k}+1}\dotsb f_{n_{k+1}}=I$, for $k>k_0$, where $k_0$ is some positive integer. So $f_n$ lies in the group part of $S$ whenever $n>n_{k_0}$. 
\end{proof}

Let us now use Theorem~\ref{kkl}  to prove Theorem~\ref{cof}.

\begin{proof}[P\textbf{roof of Theorem~\ref{cof}}]
The equivalence of~{\ref{aaa}} and~{\ref{aac}} was established already in Theorem~\ref{acj}. That~{\ref{aab}} implies~{\ref{aaa}} is immediate. The most substantial part of Theorem~\ref{cof} is the implication of~{\ref{aab}} from~{\ref{aac}}.

Suppose then that $S$ is semidiscrete and inverse free. We must prove that every composition sequence generated by $\mathcal{F}$ converges ideally. By the equivalence of~{\ref{aaa}} and~{\ref{aac}}, we know that every composition sequence generated by $\mathcal{F}$ is an escaping sequence. Theorem~\ref{kkl} then tells us that every such sequence converges ideally, provided $|\Lambda^-(S)|\neq 1$. 

When $|\Lambda^-(S)|=1$, the semigroup $S$ is elementary, and by applying Theorem~\ref{mme} (case~\ref{cle}) we can assume (after conjugating) that each element of $\mathcal{F}$ has the form $az+b$, where $0<a<1$, or $z+c$, where all such numbers $c$ have the same sign (positive or negative). Then a direct calculation shows that any composition sequence generated by $\mathcal{F}$ converges ideally; specifically, the composition sequence $F_n=f_1\dotsb f_n$, where $f_i(z)=u_iz+v_i$ belongs to  $\mathcal{F}$, converges ideally to the point $\sum_{i=1}^\infty u_1\dotsb u_{i-1}v_i$.
\end{proof}

\section{Two-generator semigroups}\label{ibj}\label{oor}

In this section we prove Theorems~\ref{piv} and~\ref{kad}. Our methods are close in spirit to those of Gilman and Maskit \cite{Gi1988,Gi1991,Gi1995,GiMa1991}. However, a difference between the two approaches is that whereas Gilman and Maskit make good use of J{\o}rgensen's inequality in their classification of two-generator discrete groups, we work the other way round: we use our classification of two-generator semidiscrete semigroups to prove a version of J{\o}rgensen's inequality for semigroups, Theorem~\ref{fff}. 

Let us begin by considering \emph{elementary} two-generator semigroups. Theorems~\ref{piv} and~\ref{kad} are not concerned with such semigroups, so we restrict ourselves to a concise summary of two-generator elementary semigroups, without proof, in the following theorem.

\begin{theorem}\label{qhw}
The semigroup $\langle f,g\rangle$ generated by two nonidentity M\"obius transformations $f$ and $g$ is elementary if and only if one of the following holds:
\begin{enumerate}
\item\label{jja} $f$ and $g$ are elliptic, and either both are of order two, or else they share a fixed point
\item\label{jjb} one of $f$ or $g$ is elliptic of order two and the other is hyperbolic, and the elliptic interchanges the fixed points of the hyperbolic
\item\label{jjc} $f$ and $g$ are parabolic or hyperbolic with at least one fixed point in common.\qedhere
\end{enumerate}
\end{theorem}

We omit the straightforward proof of this result. Using this theorem, one can determine which two-generator  elementary  semigroups are semidiscrete and which are semidiscrete and inverse free. For example, in case~\ref{jja} the semigroup $S=\langle f,g\rangle$ is semidiscrete if and only if $f$ and $g$ are both of finite order, in which case $S$ is a group, so it can never be both semidiscrete and inverse free. In case~\ref{jjb}, $S$ is again a discrete group, so it is semidiscrete but not inverse free. In case~\ref{jjc}, if $f$ and $g$ are parabolic with a common fixed point, or if $f$ and $g$ are hyperbolic with a pair of common fixed points, then $S$ may be semidiscrete and inverse free, or it may be a discrete group, or it may not be semidiscrete (see Lemma~\ref{zps} and Corollary~\ref{nps}). The remaining semigroups in case~\ref{jjc} are all semidiscrete and inverse free. 

This completes our discussion of two-generator elementary semigroups. Let us now set about proving Theorem~\ref{piv}.

Given a hyperbolic M\"obius transformation $f$, we define, as usual, $\alpha_f$ and $\beta_f$ to be the attracting and repelling fixed points of $f$, respectively. If $f$ is parabolic, then we define $\alpha_f$ and $\beta_f$ to both be the fixed point of $f$. Suppose now that $f$ and $g$ are two M\"obius transformations, each either parabolic or hyperbolic, with no fixed points in common. We say that $f$ and $g$ are \emph{antiparallel} if neither one of the two closed intervals with end points $\alpha_f$ and $\alpha_g$ is mapped strictly inside itself by both $f$ and $g$. If $f$ and $g$ are parabolic or hyperbolic and \emph{not} antiparallel, then $\langle f,g\rangle$  is a Schottky semigroup, so it is semidiscrete and inverse free.

Figure~\ref{fgc} shows configurations of antiparallel M\"obius transformations $f$ and $g$ in the disc model of the hyperbolic plane. Each hyperbolic transformation is represented by a directed hyperbolic line, corresponding to its directed axis, and each parabolic transformation is represented by a directed horocycle, corresponding to an invariant directed horocycle of the transformation based at its fixed point. The illustration on the far right, of two hyperbolic transformations, explains the terminology `antiparallel'.

\begin{figure}[ht]
\centering
\begin{tikzpicture}
\begin{scope}[xshift=-10\rad]
\pgfmathsetmacro{\t}{90-25}
\draw[black] (0,0) circle (\rad);
\hgline[black,directed](0,0)(-360+180+\t:180-\t:\rad);
\draw[black] (0.6*\rad,0) circle (0.4*\rad);	
\node[draw=none,fill=none] at (0.4*\rad,0*\rad){{\footnotesize  $g$}};
\node[draw=none,fill=none] at (-0.4*\rad,0*\rad){{\footnotesize $f$}};
\draw[directed,draw=none] (0.21*\rad,-0.1) --++ (280:0.1);
\end{scope}
\begin{scope}
\draw[black] (0,0) circle (\rad);
\draw[black] (0.6*\rad,0) circle (0.4*\rad);
\draw[black] (-0.6*\rad,0) circle (0.4*\rad);	
\node[draw=none,fill=none] at (0.4*\rad,0*\rad){{\footnotesize  $g$}};
\node[draw=none,fill=none] at (-0.4*\rad,0*\rad){{\footnotesize $f$}};
\draw[directed,draw=none] (-0.200*\rad,0.01) --++ (90:0.1);
\draw[directed,draw=none] (0.203*\rad,-0.01) --++ (270:0.1);
\end{scope}
\begin{scope}[xshift=10\rad]
\pgfmathsetmacro{\t}{90-25}
\draw[black] (0,0) circle (\rad);
\hgline[black,directed](0,0)(-360+180+\t:180-\t:\rad);
\hgline[black,directed](0,0)(\t:360-\t:\rad);	
\node[draw=none,fill=none] at (0.45*\rad,0*\rad){{\footnotesize  $g$}};
\node[draw=none,fill=none] at (-0.45*\rad,0*\rad){{\footnotesize $f$}};
\end{scope}
\end{tikzpicture}
\caption{Antiparallel M\"obius transformations}
\label{fgc}
\end{figure}

Figure~\ref{fgd} shows configurations of parabolic or hyperbolic M\"obius transformations $f$ and $g$ that have no common fixed points and that are not antiparallel.

\begin{figure}[ht]
\centering
\begin{tikzpicture}
\begin{scope}[xshift=-15\rad]
\pgfmathsetmacro{\t}{90-25}
\draw[black] (0,0) circle (\rad);
\hgline[black,directed](0,0)(180-\t:180+\t:\rad);
\hgline[black,directed](0,0)(\t:360-\t:\rad);	
\node[draw=none,fill=none] at (0.45*\rad,0*\rad){{\footnotesize  $g$}};
\node[draw=none,fill=none] at (-0.45*\rad,0*\rad){{\footnotesize $f$}};
\end{scope}
\begin{scope}[xshift=-5\rad]
\pgfmathsetmacro{\t}{25}
\draw[black] (0,0) circle (\rad);
\hgline[black,pdirected](0,0)(90+\t:270+\t:\rad);
\hgline[black,pdirected](0,0)(90-\t:270-\t:\rad);	
\node[draw=none,fill=none] at (0.40*\rad,-0.40*\rad){{\footnotesize  $g$}};
\node[draw=none,fill=none] at (-0.40*\rad,-0.40*\rad){{\footnotesize $f$}};
\end{scope}
\begin{scope}[xshift=5\rad]
\pgfmathsetmacro{\t}{90-25}
\draw[black] (0,0) circle (\rad);
\hgline[black,directed](0,0)(180-\t:180+\t:\rad);
\draw[black] (0.6*\rad,0) circle (0.4*\rad);	
\node[draw=none,fill=none] at (0.4*\rad,0*\rad){{\footnotesize  $g$}};
\node[draw=none,fill=none] at (-0.4*\rad,0*\rad){{\footnotesize $f$}};
\draw[directed,draw=none] (0.203*\rad,-0.03) --++ (270:0.1);
\end{scope}
\begin{scope}[xshift=15\rad]
\draw[black] (0,0) circle (\rad);
\draw[black] (0.6*\rad,0) circle (0.4*\rad);
\draw[black] (-0.6*\rad,0) circle (0.4*\rad);	
\node[draw=none,fill=none] at (0.4*\rad,0*\rad){{\footnotesize  $g$}};
\node[draw=none,fill=none] at (-0.4*\rad,0*\rad){{\footnotesize $f$}};
\draw[directed,draw=none] (-0.203*\rad,-0.03) --++ (270:0.1);
\draw[directed,draw=none] (0.203*\rad,-0.03) --++ (270:0.1);
\end{scope}
\end{tikzpicture}
\caption{Parabolic and hyperbolic transformations without common fixed points that are not antiparallel}
\label{fgd}
\end{figure}

Recall from the introduction the cross ratio
\[
C(f,g) = \frac{(\alpha_f-\alpha_g)(\beta_f-\beta_g)}{(\alpha_f-\beta_g)(\beta_f-\alpha_g)},
\]
where $f$ and $g$ are hyperbolic. The following lemma is elementary; we omit the proof.

\begin{lemma}\label{xxu}
Let $f$ and $g$ be hyperbolic M\"obius transformations with no common fixed points. The following statements are equivalent:
\begin{enumerate}
\item\label{cha} $f$ and $g$ are antiparallel
\item\label{chb} $\beta_f$ and $\beta_g$ lie in distinct components of $\overline{\mathbb{R}}\setminus\{\alpha_f,\alpha_g\}$
\item\label{chc} $C(f,g)>1$.\qedhere
\end{enumerate}
\end{lemma}

We wish to be able to determine when any two-generator semigroup $S$ is semidiscrete and inverse free. If one of the generators is the identity map or  elliptic, then $S$ is not semidiscrete and inverse free. On the other hand, if $f$ and $g$ are not the identity map or elliptic (they are parabolic or hyperbolic), and they are not antiparallel, then $S$ \emph{is} semidiscrete and inverse free. The remaining more difficult case is when $f$ and $g$ are antiparallel parabolic or hyperbolic maps. This case is handled in the following series of results. In these results we use the straightforward observation that if $\langle f,g\rangle$ is nonelementary and  is contained in a Schottky group of rank two, then it is semidiscrete and inverse free.

\begin{theorem}\label{acu}
Suppose that $f$ and $g$ are parabolic M\"obius transformations such that $\langle f,g\rangle$ is nonelementary. Then $\langle f,g\rangle$ is semidiscrete and inverse free if and only if $fg$ is not elliptic. Furthermore, if $f$ and $g$ are antiparallel, and $\langle f,g\rangle$  is semidiscrete and inverse free, then it is contained in a Schottky group of rank two. 
\end{theorem}
\begin{proof}
The theorem holds trivially if $f$ and $g$ are not antiparallel, so let us assume that they are antiparallel.

By conjugation, we can assume that $f(z)=z+2$ (with fixed point $\infty$) and $g$ has fixed point $0$. For $k=0,1$, let $\ell_k$ denote the vertical hyperbolic line between $k$ and $\infty$, and let $\sigma_k$ denote reflection in $\ell_k$. Observe that $f=\sigma_1\sigma_0$. We can write $g=\sigma_0\sigma$, where $\sigma$ is the reflection in a hyperbolic line $\ell$ that has endpoints $0$ and $a$, where $a$ is a nonzero real number. In fact, we must have $a>0$, because if $a<0$, then $f$ and $g$ both map $[0,+\infty]$ strictly inside itself, contrary to our assumption that they are antiparallel. 

Observe that $fg=\sigma_1\sigma$. If $a>1$, then $\ell_1$ and $\ell$ intersect, so $fg$ is elliptic, and $\langle f,g\rangle$ is not semidiscrete and inverse free. If $a\leq 1$, then $\ell_1$ and $\ell$ do not intersect, so $fg$ is not elliptic. Moreover, the group generated as a group by $f$ and $g$ is a Schottky group of rank two because $f$ maps the complement of the interval $[-\infty,-1]$ to $(1,+\infty)$, and $g$ maps the complement of the interval $[0,a]$ to $(-a,0)$. Therefore $\langle f,g\rangle$ is semidiscrete and inverse free.
\end{proof}

Recall that $\tr(f)=|a+d|$, where $f(z)=(az+b)/(cz+d)$ and $ad-bc=1$. Also, recall that the commutator of $f$ and $g$ is $[f,g]=fgf^{-1}g^{-1}$.

\begin{theorem}\label{acd}
Suppose that $f$ and $g$ are parabolic and hyperbolic M\"obius transformations, respectively,  such that $\langle f,g\rangle$ is nonelementary.  Then $\tr[f,g]>2$ and $\langle f,g\rangle$ is semidiscrete and inverse free if and only if $f^ng$ is not elliptic for $1\leq n \leq 2\tr(g)/\sqrt{\tr[f,g]-2}$. Furthermore, if $f$ and $g$ are antiparallel, and $\langle f,g\rangle$  is semidiscrete and inverse free, then it is contained in a Schottky group of rank two. 
\end{theorem}
\begin{proof}
Once again, we can assume that $f$ and $g$ are antiparallel, because the theorem is trivial otherwise.

For $k=0,1,\dotsc$, let $\ell_k$ denote the vertical hyperbolic line between $k$ and $\infty$, and let $\sigma_k$ denote reflection in $\ell_k$. By conjugation we can assume that $f(z)=z+2$ and $g=\sigma_0\sigma$, where $\sigma$ is the reflection in the hyperbolic line $\ell$ with end points $u$ and $v$, where $u<v$, and $u$ and $v$ have the same sign. In fact, $u$ and $v$ are positive, because if they were negative, then $f$ and $g$ would both map $[\alpha_g,+\infty]$ strictly inside itself, contrary to our assumption that they are antiparallel. 
 
Observe that $f=\sigma_k\sigma_{k-1}$ for any positive integer $k$, so $f^k=\sigma_k\sigma_0$ and $f^kg=\sigma_k\sigma$. Then $\sigma_0(z)=-\bar{z}$ and $\sigma(z) = ((u+v)\bar{z}-2uv)/(2\bar{z}-(u+v))$, so
\[
g(z)=\frac{-(u+v)z+2uv}{2z-(u+v)}.
\]
One can check that $\tr(g)=2(u+v)/(v-u)$ and $\tr [f,g]-2=16/(v-u)^2$. Hence $\tr[f,g]>2$ and
\[
\frac{2\tr(g)}{\sqrt{\tr[f,g]-2}}=u+v.
\]

Suppose that $f^ng$ is elliptic for some integer $n$ with $1\leq n \leq u+v$. Then $\langle f,g\rangle$ is not semidiscrete and inverse free. Conversely, suppose that $f^ng$ is not elliptic for $1\leq n \leq u+v$. Notice that if $u<n<v$ for some integer $n$, then certainly $n\leq u+v$, and furthermore $\ell_n$ and $\ell$ intersect, so $f^ng$ is elliptic, contrary to our assumption. Therefore $u,v\in[n,n+1]$, for some nonnegative integer $n$. In this case, the group $G$ generated as a group by $f$ and $g$ is a Schottky group. To see this, observe that $f$ maps the complement of the interval $[-\infty,n-1]$ to $(n+1,+\infty)$, and $f^{n}g$ maps the complement of the interval $[u,v]$ to $(2n-v,2n-u)$. Therefore the group generated by $f$ and $f^{n}g$ is a Schottky group of rank two, and this group coincides with $G$. We deduce that  $\langle f,g\rangle$ is semidiscrete (because $G$ is discrete) and $\langle f,g\rangle$ is inverse free (because $f$ and $g$ generate as a group a free group).  
\end{proof}

It remains only to consider semigroups generated by two antiparallel hyperbolic M\"obius transformations $f$ and $g$. In this case, define $\gamma_f$ and $\gamma_g$ to be the axes of $f$ and $g$, respectively. By Lemma~\ref{xxu}, they are disjoint. Let $\ell$ be the unique hyperbolic line that is orthogonal to $\gamma_f$ and $\gamma_g$, and let $\sigma$ denote reflection in $\ell$. Define $\sigma_f= f \sigma$ and $\sigma_g =  \sigma g$, both reflections, and let $\ell_f$ and $\ell_g$ be the lines of reflection of $\sigma_f$ and $\sigma_g$, respectively.

\begin{lemma}\label{cac}
Let $f$ and $g$ be antiparallel hyperbolic M\"obius transformations. Suppose that $\ell_f$ does not intersect $\gamma_g$ and $\ell_g$ does not intersect $\gamma_f$. Then $\langle f,g\rangle$ is semidiscrete and inverse free if and only if $fg$ is not elliptic. Furthermore, if $\langle f,g\rangle$  is semidiscrete and inverse free, then it is contained in a Schottky group of rank two. 
\end{lemma}
\begin{proof}
Observe that $fg=\sigma_f\sigma_g$. If $\ell_f$ and $\ell_g$ intersect, then $fg$ is elliptic, so $\langle f,g\rangle$ is not semidiscrete and inverse free. If $\ell_f$ and $\ell_g$ do not intersect, then $fg$ is not elliptic. In this case, let $J_f$ be the open interval in $\overline{\mathbb{R}}$ containing $\alpha_f$ that has the same endpoints as $\ell_f$. Likewise, let $J_g$ be the open interval containing $\beta_g$ that has the same endpoints as $\ell_g$. The intervals $J_f$, $\sigma(J_f)$, $J_g$, and $\sigma(J_g)$ are disjoint. Furthermore, $f$ maps the exterior of $\overline{\sigma(J_f)}$ to $J_f$, and $g$ maps the exterior of $\overline{J_g}$ to $\sigma(J_g)$. Therefore $f$ and $g$ generate as a group a Schottky group of rank two, so $\langle f,g\rangle$, which is contained in this Schottky group, is semidiscrete and inverse free. 
\end{proof}

The next two lemmas are heavily influenced by work of Gilman and Maskit \cite{Gi1988,Gi1991,GiMa1991}. 

Before we state these lemmas, we make some preliminary, elementary remarks about traces of matrices. Let $\Tr(A)$ denote the trace of a square matrix $A$. Each M\"obius transformation $f$ lifts to two matrices $\pm A$ in $\text{SL}(2,\mathbb{R})$.  If $f(z)=(az+b)/(cz+d)$, with $ad-bc=1$, then $\tr(f)=|a+d|$. It follows that $\tr(f)$ is equal to one of $\Tr(A)$ or $\Tr(-A)$, whichever is nonnegative.

\begin{lemma}\label{cuc}
Let $f$ and $g$ be antiparallel hyperbolic M\"obius transformations. Then $\tr[f,g]>2$, and  $\ell_f$ intersects $\gamma_g$ if and only if $\tr[f,g]<\tr(g)^2-2$. 

Furthermore, if $F$ and $G$ are lifts to $\textnormal{SL}(2,\mathbb{R})$ of $f$ and $g$, respectively, and both matrices have positive trace, then the condition $\tr[f,g]<\tr(g)^2-2$ implies that $FG$ has positive trace too.
\end{lemma}
\begin{proof}
By conjugation we can assume that $f(z)=\lambda z$, where $0<\lambda<1$, and $g$ has attracting fixed point $1$ and repelling fixed point $a$, where $0<a<1$. We can write $g$ in the form 
\[
g(z) = \frac{(\mu-a)z+(a-a\mu)}{(\mu-1)z+(1-a\mu)},
\]
where $\mu>1$, as one can easily check by verifying that $g(1)=1$, $g(a)=a$, $g(\infty)>1$, and $\tr(g)^2=2+\mu+\mu^{-1}$. One can also check that
\[
\tr [f,g] = 2 + \frac{(-2+\lambda+\lambda^{-1})(-2+\mu+\mu^{-1})}{-2+a+a^{-1}}.
\]
Hence $\tr[f,g]>2$. Also, we see that $\tr[f,g]<\tr(g)^2-2$ if and only if $-2+\lambda+\lambda^{-1}<-2+a+a^{-1}$, and this is so if and only if $\lambda>a$.

Now, a quick check shows that the end points of $\ell$ are $\pm \sqrt{a}$ and the end points of $\ell_f$ are $\pm \sqrt{\lambda a}$. Therefore $\ell_f$ intersects $\gamma_g$ if and only if $\sqrt{\lambda a}>a$, that is,  if and only if $\lambda>a$. 

It remains to prove the last part of the lemma. One can check that 
\[
\Tr(FG)=\frac{\lambda\mu+1-a(\lambda+\mu)}{\sqrt{\lambda\mu}(1-a)}.
\]
If $\tr[f,g]<\tr(g)^2-2$, then $\lambda>a$, so 
\[
\lambda\mu+1-a(\lambda+\mu)>\lambda\mu+1-\lambda(\lambda+\mu)=1-\lambda^2>0.
\]
Hence $\Tr(FG)>0$, as required.
\end{proof}

\begin{lemma}\label{hos}
Let $f$ and $g$ be antiparallel hyperbolic M\"obius transformations with $\tr(f)\leq \tr(g)$. Suppose that $\ell_f$ intersects $\gamma_g$, but $\ell_f$ and $\ell_g$ do not intersect. Then
\[
\tr(g)-\tr(fg) > \tr(f)-2\quad\text{and}\quad \tr(g)-\tr(f) > \tr(fg)-2.
\]
Furthermore, $f$ and $fg$ are antiparallel parabolic or hyperbolic transformations.
\end{lemma}
\begin{proof}
The second inequality in the first part of the lemma is a rearrangement of the first inequality, so we focus on the first.

By conjugation we can assume that $f(z)=\lambda^2z$, where $0<\lambda<1$. We can also assume that $\ell$ is the upper half of the unit circle, and, after conjugating by $z\mapsto -z$ if necessary, we can assume that the fixed points of $g$ are positive. The end points of $\ell_f$ are $\pm \lambda$. Let $u$ and $v$ be the end points of $\ell_g$, with $u<v$. Observe that $-\lambda< u<v\leq \lambda$ because $\tr(f)\leq \tr(g)$ (so the translation length of $f$ is less than or equal to  that of $g$) and $\ell_f$ and $\ell_g$ do not intersect.

As $g=\sigma\sigma_g$, where $\sigma(z)=1/\bar{z}$ and $\sigma_g(z)=((u+v)\bar{z}-2uv)/(2\bar{z}-(u+v))$, we have 
\[
g(z) = \frac{2z-(u+v)}{(u+v)z-2uv}.
\]
Furthermore, 
\[
\tr(f) = \lambda+\frac{1}{\lambda},\quad \tr(g) = \frac{2(1-uv)}{v-u},\quad\text{and}\quad \tr(fg)=\frac{2(\lambda^2-uv)}{\lambda(v-u)}.
\]
Hence 
\[
\tr(g)-\tr(fg) = \frac{2(1-\lambda)(uv+\lambda)}{\lambda(v-u)}.
\]

Now we split the argument into two cases. Suppose first that $u\geq 0$. Then 
\[
\tr(g)-\tr(fg) \geq \frac{2(1-\lambda)(uv+v)}{\lambda v}=\frac{2(1-\lambda)(u+1)}{\lambda}>\frac{(1-\lambda)^2}{\lambda}=\tr(f)-2.
\]
Suppose now that $u<0$. Then we must have $-\lambda<u<0<v\leq\lambda$. Therefore
\[
\tr(g)-\tr(fg) > \frac{2(1-\lambda)(\lambda-\lambda^2)}{\lambda (v-u)}>\frac{(1-\lambda)^2}{\lambda}=\tr(f)-2.
\]

Let us turn to the second part of the lemma. As $\ell_f$ and $\ell_g$ do not intersect, we see that $fg$ is not elliptic. It is parabolic if $v=\lambda$ and hyperbolic otherwise. In both cases one can easily check that $f$ and $fg$ are antiparallel; for example, in the second case, the repelling fixed point of $fg$ lies between $\tfrac12(u+v)$ and $v$ and the attracting fixed point is greater than $\lambda$, so $C(f,fg)=\alpha_{fg}/\beta_{fg}>1$.
\end{proof}

Suppose that $f$ and $g$ are antiparallel hyperbolic transformations with $\tr(f)\leq \tr(g)$ and $\tr[f,g]<\tr(g)^2-2$. Note that $\tr [f,g]=\tr[g,f]$. Let
\[
\Phi\big((f,g)\big)=
\begin{cases}
(f,fg), & \text{if $\tr(f)\leq\tr(fg)$},\\
(fg,f), & \text{if $\tr(fg)<\tr(f)$}.
\end{cases}
\]
We define a sequence of pairs of M\"obius transformations $(f_n,g_n)$, $n=0,1,\dotsc$, by $(f_0,g_0)=(f,g)$ and $(f_{n+1},g_{n+1})=\Phi\big((f_n,g_n)\big)$. The sequence terminates if we reach a pair $(f_n,g_n)$ for which $f_n$ and $g_n$ are  \emph{not} antiparallel hyperbolic transformations with $\tr[f_n,g_n]<\tr(g_n)^2-2$.

\begin{lemma}\label{zzo}
The sequence $(f_n,g_n)$, $n=0,1,\dotsc$, terminates, for any starting pair $(f,g)$.
\end{lemma}
\begin{proof}
Suppose, on the contrary, that the sequence does not terminate. Observe that all the numbers $\tr[f_n,g_n]$ are equal to some constant $s$, and by Lemma~\ref{cuc} we know that $s>2$. Next, observe that $s<\tr(g_n)^2-2$ for all $n$, so $\tr(g_n)>t$ for all $n$, where $t=\sqrt{s+2}>2$. 

Lemma~\ref{hos} tells us that $\tr(g_n)-\tr(g_{n+1})>\tr(f_{n+1})-2$, for $n=0,1,\dotsc$. It follows that the sequence $(\tr(g_n))$ is decreasing, with limit $l$, say, where $l>2$. Also, we have $\tr(g_0)-\tr(g_n) > \sum_{j=1}^n (\tr(f_j)-2)$, from which we deduce that $\tr(f_n)\to 2$ as $n\to\infty$. 

Now, for each $n$, define $F_n$ and $G_n$ to be lifts of $f_n$ and $g_n$ to $\text{SL}(2,\mathbb{R})$, respectively, such that both matrices have positive trace. From a well-known trace identity (see, for example,  \cite[Lemma~1.5.6]{Ma2016}) we have
\[
\Tr(F_nG_nF_n^{-1}G_n^{-1}) = \Tr(F_n)^2+\Tr(G_n)^2+\Tr(F_nG_n)^2-\Tr(F_n)\Tr(G_n)\Tr(F_nG_n)-2.
\]
The left-hand side of this equation is equal to either $s$ or $-s$. Let us look at the right-hand side. By the last part of Lemma~\ref{cuc}, we know that $\Tr(F_nG_n)>0$. Therefore the right-hand side is 
\[
\tr(f_n)^2+\tr(g_n)^2+\tr(f_ng_n)^2-\tr(f_n)\tr(g_n)\tr(f_ng_n)-2.
\]
Since $\tr(f_n)\to 2$ and $\tr(g_n)\to l$, where $l>2$, we see that $g_{n+1}=f_ng_n$, for sufficiently large $n$ (because $\tr(f_{n+1})\leq \tr(g_{n+1}))$. Hence $\tr(f_ng_n)\to l$, also. It follows that the right-hand side of the equation converges to $4+l^2+l^2-2l^2-2=2$ as $n\to\infty$, which is a contradiction. Thus, contrary to our assumption, the sequence $(f_n,g_n)$, $n=0,1,\dotsc$, must terminate after all.
\end{proof}

Finally we can prove Theorem~\ref{piv}. 

\begin{proof}[P\textbf{roof of Theorem~\ref{piv}}]
We recall that $\tr(f)\leq\tr(g)$ and $\langle f,g\rangle$ is nonelementary. The theorem is trivially true if either $f$ or $g$ is elliptic, or if $f$ and $g$ are hyperbolic and not antiparallel (in the first case $\langle f,g\rangle$ is not semidiscrete and inverse free, and in the second case it is semidiscrete and inverse free).  When   $f$ is parabolic, the theorem, including the last part, follows from Theorems~\ref{acu} and~\ref{acd}. When $f$ and $g$ are antiparallel hyperbolic maps, and $\tr[f,g]\geq \tr(g)^2-2$, then the result follows from Lemmas~\ref{cac} and~\ref{cuc}.

The remaining case is that $f$ and $g$ are antiparallel hyperbolic maps and $\tr[f,g]<\tr(g)^2-2$. Let us form the sequence $(f_n,g_n)$, $n=0,1,\dotsc$, with $(f_0,g_0)=(f,g)$. Lemma~\ref{zzo} shows us that we must eventually reach a terminal pair $(f_m,g_m)$. By Lemma~\ref{hos}, $f_m$ and $g_m$ are antiparallel, so either one of them is parabolic, or otherwise they are both hyperbolic but $\tr[f_m,g_m]\geq \tr(g_m)^2-2$. Either way, we see from the earlier cases that either $\langle f_m,g_m\rangle$ is not semidiscrete and inverse free, or else it is contained in a Schottky group of rank two. If the former holds, then because $\langle f_m,g_m\rangle$ is contained in both $\langle f,g\rangle$ and $\langle f,fg\rangle$, we see that these two semigroups are not semidiscrete and inverse free either. If the latter holds, then because $f_k$ and $g_k$ generate as a group the same group no matter the index $k$, we see that $\langle f,g\rangle$ and $\langle f,fg\rangle$ are both contained in a Schottky group of rank two as well, so both are semidiscrete and inverse free. This completes the proof.
\end{proof}

So far in our study of two-generator nonelementary semigroups, we have ignored the possibility that one of the generators is elliptic, because if a semigroup has an elliptic generator, then it is not semidiscrete and inverse free. This gap in our study is filled by the theorem to follow shortly, after the next, handy lemma.

\begin{lemma}\label{qop}
Suppose that the two-generator semigroup $\langle f,g\rangle$ contains an element $w=w_1\dotsb w_n$ of finite order, where $w_j\in\{f,g\}$, for all $j$, and the maps $w_j$ are not all equal. Then $\langle f,g\rangle$ coincides with the group generated as a group by $f$ and $g$.
\end{lemma}
\begin{proof}
By replacing $w$ with a positive power of $w$, we can assume that $w=I$, the identity map. Then $w_k\dotsb w_nw_1\dotsb w_{k-1}=I$, for $k=1,\dotsc,n$, so we see that $w_k^{-1}\in\langle f,g\rangle$. In particular, both $f^{-1}$ and $g^{-1}$ belong to $\langle f,g\rangle$, so $\langle f,g\rangle$ is a group.
\end{proof}

Recall that $\text{ord}(f)$ denotes the order of an elliptic M\"obius transformation of finite order.

\begin{theorem}\label{cba}
Suppose that $f$ and $g$ are M\"obius transformations, neither the identity map, and $f$ is elliptic of finite order. Let $G$ be the group generated as a group by $f$ and $g$. Then $\langle f,g\rangle$ is semidiscrete if and only if exactly one of the following holds:
\begin{enumerate}
\item\label{ela} $g$ is elliptic of finite order, and $G$ is discrete 
\item\label{elb} $g$ is not elliptic and $f^ng$ has finite order for some positive integer $n$, and $G$ is discrete, or
\item\label{elc} $g$ is not elliptic and $f^ng$ is not elliptic for $n=1,\dots,\textnormal{ord}(f)-1$.
\end{enumerate}
In cases \ref{ela} and \ref{elb} we have $\langle f,g\rangle = G$, and in case \ref{elc} $\langle f,g\rangle$ has nonempty group part and nonempty inverse-free part, and it is contained in a discrete group. Furthermore, every semidiscrete semigroup generated by two nonidentity maps that has nonempty group part and nonempty inverse-free part is of type \ref{elc}.
\end{theorem}

The proof of Theorem~\ref{cba} uses the concept of an \emph{anticonformal M\"obius transformation}, which is a map of the form $z\mapsto (a\overline{z}+b)/(c\overline{z}+d)$, where $a,b,c,d\in\mathbb{R}$ and $ad-bc<0$. The collection of all such maps is the full collection $\mathcal{A}$ of anticonformal isometries of $\mathbb{H}$. (Actually, the proof uses the unit disc $\mathbb{D}$ model of the hyperbolic plane -- the collection of anticonformal M\"obius transformation of $\mathbb{D}$ is obtained by conjugating $\mathcal{A}$ by a complex M\"obius transformation that takes $\mathbb{H}$ to $\mathbb{D}$.)

\begin{proof}[P\textbf{roof of Theorem~\ref{cba}}]
Suppose first that $g$ is elliptic of finite order. Then $\langle f,g\rangle$ is equal to $G$, so $\langle f,g\rangle$ is semidiscrete if and only if $G$ is discrete. The theorem also holds when $g$ is an elliptic of infinite order, trivially.

Now suppose that $g$ is not elliptic. Suppose also that $f^ng$ has finite order for some positive integer $n$. It follows from Lemma~\ref{qop} that $\langle f,g\rangle$ is equal to $G$, so, again, $\langle f,g\rangle$ is semidiscrete if and only if $G$ is discrete.

Next we consider the case when $g$ is not elliptic and $f^ng$ has infinite order for every positive integer $n$. Since $f$ has finite order ($m$, say), we need only concern ourselves with those values of $n$ from $1$ to $m-1$. In fact, we can assume that $f^ng$ is not elliptic for $n=1,\dots,m-1$, as otherwise -- if one of the maps $f^ng$ is elliptic of infinite order -- the theorem holds, trivially. We will prove that $\langle f,g\rangle$ is semidiscrete. Consider the action of $f$ and $g$ on the unit disc $\mathbb{D}$, for a change.  By conjugation, we can assume that $0$ is the fixed point of $f$. We can also assume that if $g$ is hyperbolic, then its axis lies in the right half of $\mathbb{D}$ (or it is the line between $-i$ and $i$) and is symmetric about the real axis, and if $g$ is parabolic, then its fixed point is $1$. Furthermore, by replacing $f$ with $f^d$, where $d$ and $m$ are coprime, we can assume that $f(z)=e^{2\pi i/m}z$. 

Next, for $k=0,1,\dots,m-1$, let $\ell_k$ denote the hyperbolic line between $\pm e^{\pi ik/m}$, and let $\sigma_k$ denote reflection in $\ell_k$. Then $f^k=\sigma_k\sigma_0$. Let $\sigma=\sigma_0g$, so $g=\sigma_0\sigma$. Then $\sigma$ is a reflection in some hyperbolic line $\ell$. This line cannot intersect any of the lines $\ell_k$ because if it did, then $f^kg=\sigma_k\sigma$ would be elliptic. Therefore $\ell$ must lie in one of the sectors between consecutive lines $\ell_k$ and $\ell_{k+1}$ (or $\ell_{m-1}$ and $\ell_0$), possibly meeting one or both of these lines on the ideal boundary.

The region enclosed by $\ell$, $\ell_k$, and $\ell_{k+1}$ (shown in Figure~\ref{ihx}) is a fundamental region for the discrete group of conformal and anticonformal M\"obius transformations generated as a group by $\sigma$, $\sigma_k$, and $\sigma_{k+1}$. This group contains $\langle f,g\rangle$, so $\langle f,g\rangle$ is discrete, and hence semidiscrete.

\begin{figure}[ht]
	\centering
	\begin{tikzpicture}[scale=1]
	\draw[draw=none,fill=col1] (0,0) -- (240:\Rad) arc (240:253:\Rad) -- ($(0,0)+({\Rad*cos(253)},{\Rad*sin(253)})$) arc (163:23:{\Rad*tan(20)}) -- (293:\Rad) arc (293:300:\Rad) -- cycle;
	\draw[black] (0,0) circle (\Rad);
	\hgline[black](0,0)(60:240:\Rad);
	\hgline[black](0,0)(120:300:\Rad);		
	\hgline[black](0,0)(253:293:\Rad);
	\node[draw=none,fill=none] at (0.05*\Rad,-0.8*\Rad){{\footnotesize  $\ell$}};
	\node[draw=none,fill=none] at (-0.4*\Rad,-0.45*\Rad){{\footnotesize $\ell_k$}};
	\node[draw=none,fill=none] at (0.45*\Rad,-0.45*\Rad){{\footnotesize $\ell_{k+1}$}};	
	\node[draw=none,fill=none] at (0,-0.18*\Rad){{\footnotesize $\tfrac{\pi}{m}$}};
	\draw[black] (240:0.3*\Rad) arc (240:300:0.3*\Rad);	
	\end{tikzpicture}
	\caption{Fundamental region for the group generated as a group by $\sigma$, $\sigma_k$, and $\sigma_{k+1}$}
	\label{ihx}
\end{figure}

The semigroup  $\langle f,g\rangle$ clearly has nonempty group part. Let us prove that it has nonempty inverse-free part. To do this, consider the two closed intervals on the unit circle 
\[
\{e^{\pi i\theta/m}\,: k\leq \theta\leq k+1\}\quad\text{and}\quad\{e^{\pi i\theta/m}\,: -(k+1)\leq \theta\leq -k\},
\]
and define $Y$ to be whichever of the two intervals contains both end points of $\ell$. Let $X=Y\cup f(Y)\cup \dotsb \cup f^{m-1}(Y)$. Then $X$ is a nontrivial closed subset of the unit circle that satisfies $f(X)=X$ and $g^{-1}(X)\subset Y$, so each element of the semigroup $\langle f,g^{-1}\rangle$ maps $X$ inside itself. 

Observe that $g(\sigma_0(X)) =\sigma_0(\sigma(\sigma_0(X)))\subset \sigma_0(X)$. If $g\in \langle f,g^{-1}\rangle$, then $g(X)\subset X$, so $g$ maps the finite set $X\cap \sigma_0(X)$ inside itself, which is false. Thus, $g\notin \langle f,g^{-1}\rangle$ after all. Consequently $g^{-1}\notin\langle f,g\rangle$, so $\langle f, g\rangle$ has nonempty inverse-free part.

It remains to show that every semidiscrete semigroup generated by two nonidentity maps that has nonempty group part and nonempty inverse-free part is of type (iii). Suppose that $\langle f,g\rangle$ is such a semigroup. Lemma~\ref{qop} tells us that one of $f$ or $g$ must be elliptic of finite order -- say $f$. Then the earlier parts of the theorem tell us that $\langle f,g\rangle$ must be of type (iii).
\end{proof}

Theorem~\ref{kad} follows immediately from Theorem~\ref{cba}, because if $\langle f,g\rangle$ is a nonelementary semidiscrete semigroup that is not inverse free, and not a group, then it must have nonempty group part and nonempty inverse-free part, which implies that it must be of class~\ref{elc} from Theorem~\ref{cba}.

Let us finish this section by proving a version of J{\o}rgensen's inequality for semidiscrete semigroups.

\begin{theorem}\label{fff}
Suppose that $f$ and $g$ are M\"obius transformations such that $\langle f,g\rangle$ is nonelementary and semidiscrete. Then either $f$ and $g$ both map a closed interval strictly inside itself, or else
\[
\lvert \Tr(F)^2-4 \rvert+\lvert\Tr[F,G]-2\rvert\geq 1,
\]
where $F$ and $G$ are lifts of $f$ and $g$, respectively, to $\textnormal{SL}(2,\mathbb{R})$.
\end{theorem}

Equality is achieved in the trace inequality for the maps $f(z)=z+1$ and $g(z)=-1/z$, which generate the modular group (even as a semigroup).

\begin{proof}
Suppose that $f$ and $g$ do not both map a closed interval strictly inside itself. By Theorem~\ref{kad}, we know that either $\langle f,g\rangle$ is semidiscrete and inverse free, or else it is contained in a discrete group. In the former case, by Theorem~\ref{piv}, $\langle f,g\rangle$ is contained in a Schottky group of rank two. Therefore in all cases $\langle f,g\rangle$ is contained in a discrete group, so J{\o}rgensen's inequality $\lvert\Tr(F)^2-4\rvert+\lvert\Tr[F,G]-2\rvert\geq 1$ applies from the theory of discrete groups. 
\end{proof}

\section{Classification of semigroups}\label{oaq}

In this section we prove Theorem~\ref{ppd} and various other results. Our first theorem has a well-known counterpart in the theory of Fuchsian groups. Recall that $\alpha_f$ and $\beta_f$ denote the attracting and repelling fixed points of a hyperbolic map $f$, respectively.

\begin{theorem}\label{cca}
Suppose that $S$ is a nonelementary semigroup. Then for every pair of open subsets $U$ and $V$ of $\overline{\mathbb{R}}$ such that $U$ meets $\Lambda^+(S)$ and $V$ meets $\Lambda^-(S)$, there is a hyperbolic element of $S$ with attracting fixed point in $U$ and repelling fixed point in $V$.
\end{theorem}
\begin{proof}
By Theorem~\ref{aiu}, there are hyperbolic maps $f$ and $g$ in $S$ such that $\alpha_f\in U$ and $\beta_g\in V$, and since $\Lambda^+(S)$ is perfect we can assume that $\alpha_f\neq \beta_g$. If either $\alpha_g\in U$ or $\beta_f\in V$, then we have found a hyperbolic map of the required type. Suppose instead that $\alpha_g\notin U$ and $\beta_f\notin V$. For the moment, let us assume also that $\beta_f\neq \alpha_g$ (so no two of $\alpha_f$, $\beta_f$, $\alpha_g$, and $\beta_g$ are equal). We can choose pairwise disjoint open intervals $A_f$, $B_f$, $A_g$, and $B_g$ such that $\alpha_f\in A_f$, $\beta_f\in B_f$, $\alpha_g\in A_g$, and $\beta_g\in B_g$, and also $A_f\subset U$ and $B_g\subset V$. Now let $n$ be a sufficiently large positive integer that $f^n$ maps the complement of $B_f$ into $A_f$, and $f^{-n}$ maps the complement of $A_f$ into $B_f$. Suppose also that $n$ is large enough that $g^n$ maps the complement of $B_g$ into $A_g$, and $g^{-n}$ maps the complement of $A_g$ into $B_g$. One can now check that the map $h=f^ng^n$ satisfies $h(A_f)\subset A_f$ and $h^{-1}(B_g)\subset B_g$. Hence $h$ is a hyperbolic element of $S$ with $\alpha_h\in U$ and $\beta_h\in V$, as required.

It remains to consider the case when $\beta_f=\alpha_g$ (and, as before, $\alpha_f$, $\beta_f$, and $\beta_g$ are pairwise distinct). Since $S$ is nonelementary, there is an element $s$ of $S$ that maps $\beta_f$ to a point outside the set $\{\alpha_f, \beta_f,\beta_g\}$. Also, by replacing $s$ with either $sg$ or $sf$ if necessary, we can assume that $s^{-1}(\beta_f)$ lies outside the set $\{\alpha_f,\beta_f,\beta_g\}$. We can now choose intervals $A_f$, $B_f$, $A_g$, and $B_g$ as before, but this time choose them such that $B_f=A_g$, and such that the larger collection of intervals $\{A_f, B_f,B_g,s(B_f),s^{-1}(B_f)\}$ is pairwise disjoint. Arguing in a similar way to before, we see that if  $n$ is chosen to be sufficiently large, the map $h=f^nsg^n$ is hyperbolic with $\alpha_h\in U$ and $\beta_h\in V$. 
\end{proof}

Let us now draw the reader's attention to three special families of M\"obius transformations, namely
\begin{enumerate}
\item\label{uua} $(a z -b)/(b z+a)$, $a^2+b^2=1$
\item\label{uub} $\lambda z$, $\lambda\geq 1$
\item\label{uuc} $z+\mu$, $\mu\geq 0$.
\end{enumerate}

Each of these is a one-parameter semigroup. The first consists of all elliptic rotations about the point $i$, the second consists of hyperbolic transformations with attracting fixed point $\infty$ and repelling fixed point $0$, and the third is a collection of parabolic transformations that fix $\infty$ (and the identity map is contained in each family too). We will prove that, up to conjugacy, any closed semigroup that is not semidiscrete contains one of these families. However, there is a caveat here: we must allow conjugacy by conformal \emph{or anticonformal} M\"obius transformations. For us, a \emph{conformal} M\"obius transformation is a map of the form $z\mapsto (az+b)/(cz+d)$, where $a,b,c,d\in\mathbb{R}$ and $ad-bc>0$, and, as we saw in the previous section, an \emph{anticonformal} M\"obius transformation is a map of the  form $z\mapsto (a\overline{z}+b)/(c\overline{z}+d)$, where $a,b,c,d\in\mathbb{R}$ and $ad-bc<0$. Thus far we have referred to conformal M\"obius transformations simply as M\"obius transformations, and we will continue to do so; we only introduce the term `conformal' now to emphasise the distinction with anticonformal M\"obius transformations. Of course, when you conjugate a semigroup by an anticonformal M\"obius transformation you obtain another semigroup. The reason we need anticonformal M\"obius transformations is to simplify the treatment of case~\ref{uuc}; after all, the two semigroups $\{z\mapsto z+\mu\,:\,\mu\geq 0\}$ and $\{z\mapsto z+\mu\,:\,\mu\leq 0\}$ are conjugate by $z\mapsto -\overline{z}$, but they are not conjugate in $\mathcal{M}$. In fact, the two semigroups are conjugate by the simpler map $z\mapsto -z$, because the transformation $z\mapsto \overline{z}$ commutes with real M\"obius transformations, so we may as well use this simpler conjugating map instead (as we have done already, in earlier sections).

For the next lemma, and later in this section, we write $\overline{S}$ to mean the closure of $S$ in $\mathcal{M}$.

\begin{lemma}\label{ljq}
Let $S$ be a closed semigroup that is not semidiscrete. Then $S$ is conjugate by a conformal or an anticonformal M\"obius transformation to a semigroup that contains one of the families {(i)}, {(ii)} or {(iii)}.
\end{lemma}
\begin{proof}
Since $S$ is not semidiscrete, there is a sequence $(g_n)$ in $S\setminus\{I\}$ that converges uniformly to $I$, the identity map. Suppose first that this sequence contains infinitely many hyperbolic elements. By passing to a subsequence, we can assume that every map $g_n$ is hyperbolic. Let $\alpha_n$ and $\beta_n$ be the attracting and repelling fixed points of $g_n$, respectively. By passing to a further subsequence of $(g_n)$, we can assume that the sequences $\alpha_n$ and $\beta_n$ both converge. Suppose for now that they converge to distinct values, which, after conjugating $S$ if need be, we can assume are $\infty$ and $0$, respectively. Let $(h_n)$ be any sequence of M\"obius transformations that satisfies $h_n(\infty)=\alpha_n$, $h_n(0)=\beta_n$, and $h_n\to I$ as $n\to\infty$. Define $k_n=h_n^{-1}g_nh_n$; this hyperbolic map has repelling fixed point $0$ and attracting fixed point $\infty$, so $k_n(z)=\lambda_n z$, where $\lambda_n> 1$. Furthermore, $k_n\to I$ as $n\to\infty$ (so $\lambda_n\to 1$). Now select any number $\lambda>1$, and define $f(z)=\lambda z$. By passing to yet another subsequence of $(g_n)$ if necessary,  we can assume that $\lambda_n<\lambda$, for $n=1,2,\dotsc$. For each positive integer $n$, the sequence $\lambda_n, \lambda_n^2,\lambda_n^3,\dotsc$ is strictly increasing with limit $\infty$. Define $t_n$ to be the unique positive integer such that $\lambda \in[\lambda_n^{t_n},\lambda_n^{t_n+1})$. Let $f_n=k_n^{t_n}$. Notice that $f_n(1)=\lambda_n^{t_n}$ and $f(1)=\lambda$. Also,
\[
\chi(f_n(1),f(1))\leq \chi(k_n^{t_n}(1),k_n^{t_n+1}(1))\leq \chi_0(k_n^{t_n},k_n^{t_n+1})=\chi_0(I,k_n).
\] 
Since $f_n$ fixes $0$ and $\infty$ for each $n$, we see that $f_n(1)\to f(1)$, $f_n(0)\to f(0)$, and $f_n(\infty)\to f(\infty)$ as $n\to\infty$, so $f_n\to f$. As a consequence, $g_n^{t_n}=h_nf_nh_n^{-1}\to f$ as $n\to\infty$. Therefore, in this case, the family of type~\ref{uub} is contained in our semigroup.

Near the start of the preceding argument we assumed that the sequences $(\alpha_n)$ and $(\beta_n)$ converged to \emph{distinct} values. Let us resume the argument from that point, but this time assume that $(\alpha_n)$ and $(\beta_n)$ converge to the same value, which, after conjugating $S$ if need be, we can assume is $\infty$. Let $(h_n)$ be any sequence of M\"obius transformations that satisfies $h_n(\infty)=\alpha_n$ and $h_n\to I$ as $n\to\infty$. Define $k_n=h_n^{-1}g_nh_n$; this hyperbolic map has attracting fixed point $\infty$. Let $\delta_n$ be the repelling fixed point of $k_n$. Then $\delta_n\to\infty$ as $n\to\infty$. By passing to a subsequence of $(g_n)$, we can assume that the numbers $\delta_n$ all have the same sign, which, after conjugating by $z\mapsto -z$ if need be, we can assume  is negative. We can write $k_n(z)=\lambda_n(z-\delta_n)+\delta_n$, where $\lambda_n>1$. As before, $k_n\to I$ as $n\to\infty$ (so $\lambda_n\to 1$). Now select any number $\mu>0$, and define $f(z)=z+\mu$. By passing to yet another subsequence of $(g_n)$,  we can assume that $k_n(0)=\delta_n(1-\lambda_n)<\mu$, for $n=1,2,\dotsc$. For each positive integer $n$, the sequence $k_n(0), k_n^2(0),k_n^3(0),\dotsc$ is strictly increasing with limit $\infty$. Define $t_n$ to be the unique positive integer such that $\mu\in[k_n^{t_n}(0),k_n^{t_n+1}(0))$. Let $f_n=k_n^{t_n}$. Then
\[
\chi(f_n(0),f(0))\leq \chi(k_n^{t_n}(0),k_n^{t_n+1}(0))\leq \chi_0(I,k_n).
\] 
Hence $f_n(0)\to f(0)$; that is, $\delta_n(1-\lambda_n^{t_n})\to \mu$. Since $\delta_n\to\infty$, we see that $\lambda_n^{t_n}\to 1$. Hence, for any real number~$x$, 
\[
f_n(x) = \lambda_n^{t_n}x +\delta_n(1-\lambda_n^{t_n}) \to x+\mu\quad\text{as}\quad n\to\infty.
\]
So $f_n\to f$, and, as a consequence, $g_n^{t_n}=h_nf_nh_n^{-1}\to f$ too. Therefore, in this case, the family of type~\ref{uuc} is contained in our semigroup.

We began this proof by choosing a sequence $(g_n)$ in $S\setminus\{I\}$ such that $g_n\to I$, and supposing that there were infinitely many hyperbolic maps in this sequence. If there are infinitely many parabolic maps in the sequence, then we can carry out an argument similar to those we have given already to show that a conjugate of $S$ contains the family of type~\ref{uuc}. The remaining possibility is that almost all the maps $g_n$ are elliptic, in which case we may as well assume that they are all elliptic. If any one of them is elliptic of infinite order, then we obtain the family of type~\ref{uua} (up to conjugacy) by considering the closure of the semigroup generated by that map alone. If infinitely many of the maps $g_n$ share a fixed point, then it is straightforward to see that $\overline{S}$ is conjugate to a semigroup that contains the family of type~\ref{uua}. The only other possibility is that the maps $g_n$ have infinitely many different fixed points. By passing to a subsequence of $(g_n)$, we can assume that the fixed points of the maps $g_n$ are pairwise distinct. Let $h_n=g_ng_{n+1}g_n^{-1}g_{n+1}^{-1}$. Since the maps $g_n$ are of finite order, $h_n$ is an element of $S$, and $h_n\to I$ as $n\to\infty$. Furthermore, one can easily check that $h_n$ is hyperbolic, so by the earlier arguments we see that the closure of our semigroup contains one of the families of types~\ref{uub} or~\ref{uuc}.
\end{proof}

\begin{lemma}\label{mkw}
Suppose that $g$ is a hyperbolic M\"obius transformation with $0<\alpha_g<\beta_g<+\infty$. Then in each of the families~\ref{uua}, ~\ref{uub}, and~\ref{uuc} there is a map $f$ such that $fg$ is elliptic. 
\end{lemma}
\begin{proof}
There is a simple algebraic proof of this lemma, but the geometric proof we offer is more illuminating. We consider cases~\ref{uua},~\ref{uub}, and~\ref{uuc} in turn; first case~\ref{uua}. Let $\gamma$ denote the reflection in the hyperbolic line $\ell_\gamma$ that passes through $i$ and is orthogonal to the axis of $g$. Then $g=\gamma\beta$, where $\beta$ is the reflection in another line $\ell_\beta$ that is parallel to $\ell_\gamma$. Choose any line $\ell_\alpha$ that passes through $i$ and intersects $\ell_\beta$, and let $\alpha$ be the reflection in $\ell_\alpha$. Then $f=\alpha\gamma$ is an elliptic map that fixes $i$, so it is of type~\ref{uua}. Moreover, $fg=\alpha\beta$ is also elliptic, because $\ell_\alpha$ and $\ell_\beta$ intersect in a point.

Now for case~\ref{uub}. This time let $\ell_\gamma$ be the line that is orthogonal to the axis of $g$ and to the vertical hyperbolic line $L$ between $0$ and $\infty$ (the positive imaginary axis). In Euclidean terms, $\ell_\gamma$ is a Euclidean semicircle centred on the origin that is symmetric about $L$. Outside this semicircle is the hyperbolic line $\ell_\beta$, where, as before, $g=\gamma\beta$. Let $\ell_\alpha$ be any hyperbolic line that is orthogonal to $L$ and cuts $\ell_\beta$. Then $f=\alpha\gamma$ is hyperbolic with attracting fixed point $\infty$ and repelling fixed point $0$, so it is of type~\ref{uub}. And again, $fg$ is elliptic.

Case~\ref{uuc} is similar to case~\ref{uub}, and we omit the argument. 
\end{proof}

We now set about proving Theorem~\ref{ppd}. We will need the following result of B{\'a}r{\'a}ny, Beardon, and Carne \cite[Theorem 3]{BaBeCa1996}, mentioned in the introduction.

\begin{theorem}\label{qab}
Suppose that $f$ and $g$ are noncommuting elements of $\mathcal{M}$, and one of them is an elliptic element of infinite order. Then $\langle f,g\rangle$ is dense in $\mathcal{M}$. 
\end{theorem}

We also need the following lemma, taken from \cite[Theorem~8.4.1]{Be1995}.

\begin{lemma}\label{qui}
Any nonelementary nondiscrete group of M\"obius transformations contains an elliptic element of infinite order.
\end{lemma}

The next theorem is a stronger version of Theorem~\ref{ppd} (we will need this stronger statement later). Recall that $\mathcal{M}(J)$ denotes the semigroup of those M\"obius transformations that map a closed interval $J$ (which, according to our conventions, is nontrivial and not a singleton interval) within itself.

\begin{theorem}\label{lwk}
Suppose that $S$ is not elementary, semidiscrete, or contained in $\mathcal{M}(J)$, for any  closed interval $J$. Then there is a two-generator semigroup within $S$ that is dense in $\mathcal{M}$.
\end{theorem}
\begin{proof}
By Lemma~\ref{ljq}, we can assume, after conjugating $S$  by a conformal or anticonformal M\"obius transformation, that its closure $\overline{S}$ contains one of the families of maps~\ref{uua},~\ref{uub}, or~\ref{uuc}. In order to apply Lemma~\ref{mkw}, we will prove that there is a hyperbolic map $g$ in $S$ with $0<\alpha_g<\beta_g<+\infty$ (possibly after conjugating $S$ again). To this end, suppose first that $\overline{S}$ contains the family of type (i). The sets $\Lambda^-(S)$ and $\Lambda^+(S)$ are perfect, so, using Theorem~\ref{cca}, we can choose a hyperbolic map $g$ in $S$ such that $\chi(\alpha_g,\beta_g)$ is less than $\chi(0,\infty)=2$, the maximum value of $\chi$. After conjugating $S$ by elliptic rotations about $i$, and possibly the map $z\mapsto -z$ (which fixes the family of type~\ref{uua}), we can assume that $0<\alpha_g<\beta_g<+\infty$.

Suppose now that $\overline{S}$ contains a family of type~\ref{uub}.  Let us also suppose, for the moment, that we can write $\overline{\mathbb{R}}$ as the union of two closed intervals $J$ and $K$ that have disjoint interiors, but common end points $u$ and $v$ (so $\overline{\mathbb{R}}=J\cup K$ and $|J\cap K|=\{u,v\}$), with $\Lambda^+(S)\subset J$ and $\Lambda^-(S)\subset K$. By shrinking $J$ we can assume that $u,v \in \Lambda^+(S)$, because $\Lambda^+(S)$ is closed and uncountable. Choose an element $f$ of $S$.  If $f$ is hyperbolic, then $\alpha_f\in J$ and $\beta_f\in K$, so $f(J)\subset J$. If $f$ is parabolic, then its fixed point must lie in $J\cap K$, and, because $\Lambda^+(S)$ is forward invariant under $S$, we must have $f(J)\subset J$. If $f$ is elliptic or the identity map, then, because $f(\Lambda^+(S))\subset \Lambda^+(S)$ and $\Lambda^+(S)\subset J$, $f$ must be of finite order. But if it is of finite order, then $f$ fixes both $\Lambda^+(S)$ and $\Lambda^-(S)$ as sets. It follows that $f$ must be the identity map. We have deduced that $S\subset \mathcal{M}(J)$, contrary to one of the hypotheses of the theorem.

We now see that there are no such intervals $J$ and $K$. Since $0\in\Lambda^-(S)$ and $\infty\in \Lambda^+(S)$, there must be points $p$ in $\Lambda^-(S)$ and $q$ in $\Lambda^+(S)$ with either $-\infty<p<q<0$ or $0<q<p<+\infty$. After conjugating $S$ by the map $z\mapsto -z$ (which fixes the family of type~\ref{uub}), we can assume that the points $p$ and $q$  satisfy $0<q<p<+\infty$. We can now apply Theorem~\ref{cca} to deduce the existence of a hyperbolic map $g$ in $S$ with $0<\alpha_g<\beta_g<+\infty$.

The remaining case is that $\overline{S}$ contains a family of type (iii). For this family, we can argue as before that there are points $p$ in $\Lambda^-(S)$ and $q$ in $\Lambda^+(S)$ such that $-\infty<q<p<+\infty$. After conjugating $S$ by a translation we can assume that $q>0$. Once again, we obtain a hyperbolic map $g$ with $0<\alpha_g<\beta_g<+\infty$.

We are now in a position to apply Lemma~\ref{mkw}, which gives an element $f$ of one of the families~\ref{uua},~\ref{uub}, or~\ref{uuc}, whichever one we are dealing with (the one that lies in $\overline{S}$), such that $fg$ is elliptic. By adjusting $f$ slightly we can assume that $fg$ is not of order two (because the set of elliptic maps not of order two is open in $\mathcal{M}$).

 Let $f^{1/n}$, for $n\in\mathbb{N}$, denote the $n$th compositional root of $f$ in the same family. Then $f^{1/n}\to I$ as $n\to\infty$. Let us choose $n$ to be sufficiently large that $h=f^{1/n}$ and $g$ fail to satisfy J{\o}rgensen's inequality, in the sense that if $H$ and $G$ are lifts to $\text{SL}(2,\mathbb{C})$ of $h$ and $g$, respectively, then
\[
\lvert\Tr(H)^2-4\rvert+\lvert\Tr [H,G] - 2\rvert<1.
\]
Now choose an element $h_0$ of $S$ sufficiently close to $h$ that $h_0^ng$ is elliptic not of order two, and $h_0$ and $g$ also fail to satisfy J{\o}rgensen's inequality. It follows that the group generated as a group by $h_0$ and $g$ is either elementary or not discrete. 

Suppose that $h_0^ng$ has infinite order. Then, because $h_0^ng$ and $g$ do not commute, we can apply Theorem~\ref{qab} to see that $\langle h_0^ng,g\rangle$ is dense in $S$. Otherwise, $h_0^ng$ has finite order, in which case $\langle h_0,g\rangle$ is a group, by Lemma~\ref{qop}. It is not an elementary group because it contains an elliptic map $h_0^ng$ and a hyperbolic map $g$, and the elliptic map does not interchange the fixed points of the hyperbolic map, as it is not of order two. Hence $\langle h_0,g\rangle$ is not discrete. Therefore it contains an elliptic element of infinite order, by Lemma~\ref{qui}, so Theorem~\ref{qab} tells us that $\langle h_0,g\rangle$ is dense in $\mathcal{M}$.
\end{proof}


\section{Classification of finitely-generated semigroups}\label{gaq}
 
In this section we present a version of Theorem~\ref{ppd} for finitely-generated semigroups. We use the notation $\mathcal{M}_0(J)$, introduced earlier, for the group part of $\mathcal{M}(J)$.

\begin{theorem}\label{iab}
Let $S$ be a finitely-generated semigroup. Then $S$ is 
\begin{enumerate}
\item elementary
\item semidiscrete
\item contained in $\mathcal{M}(J)$, for some interval $J$, and is either exceptional or dense in $\mathcal{M}_0(J)$, or
\item dense in $\mathcal{M}$.\qedhere
\end{enumerate}
\end{theorem}
\begin{proof}
Suppose $S$ is not elementary, semidiscrete, or dense in $\mathcal{M}$. Theorem~\ref{ppd} tells us that $S$ is contained in $\mathcal{M}(J)$, for some closed interval $J$. By conjugation, we can assume that $J=[0,+\infty]$. Let us now examine the three possible types of semigroups in  $\mathcal{M}_0(J)$ according to Corollary~\ref{nps}. In case~\ref{wea} of Corollary~\ref{nps}, $S$ is semidiscrete, by Theorem~\ref{jjr}. In case~\ref{web}, $S$ is either exceptional or semidiscrete, by Theorem~\ref{ioe}. In case~\ref{wec}, $S$ is dense in $\mathcal{M}_0(J)$. This completes our classification of finitely-generated semigroups.
\end{proof}

Our next task is to prove Theorem~\ref{abd}. This theorem is the counterpart for semigroups of the well-known result for Fuchsian groups that a nonelementary group of M\"obius transformations is discrete if and only if every two-generator subgroup is discrete. Here is a version of that result for semigroups, which, unlike Theorem~\ref{abd}, does not assume that the semigroup is finitely generated.

\begin{theorem}
Let $S$ be a nonelementary semigroup that is not contained in $\mathcal{M}(J)$, for any nontrivial closed interval $J$. Then $S$ is semidiscrete if and only if every two-generator semigroup contained in $S$ is semidiscrete.
\end{theorem}

This theorem is an immediate corollary of Theorem~\ref{lwk}. The assumption that $S$ is not contained in $\mathcal{M}(J)$ cannot be removed. To see why this is so, consider, for example, the semigroup $S$ generated by the maps $(1-1/n)z$, $n=2,3,\dotsc$. Clearly $S$ is not semidiscrete, but every two-generator semigroup within $S$ \emph{is} semidiscrete. This particular semigroup is elementary, but it is easy to adjust it to give a nonelementary example: simply adjoin to $S$ any element of $\mathcal{M}([0,+\infty])$ that fixes neither $0$ nor $\infty$.

The theorem also fails if the assumption that $S$ is not elementary is removed. For example, one can easily check that the semigroup $S=\{2^nz+b\,:\,n\in\mathbb{Z}, b\in\mathbb{Q}\}$ is not semidiscrete, but every two-generator semigroup contained in $S$ is semidiscrete.

Let us now turn to Theorem~\ref{abd}, which we restate for convenience.

\begin{reptheorem}{abd}
Any finitely-generated  nonelementary nonexceptional semigroup $S$ is semidiscrete if and only if every two-generator semigroup contained in $S$ is semidiscrete.
\end{reptheorem}
\begin{proof}
If $S$ is semidiscrete, then certainly any two-generator semigroup in $S$ is semidiscrete. Conversely, suppose that every two-generator semigroup contained in $S$ is semidiscrete. By Theorem~\ref{lwk}, $S$ is either  semidiscrete or contained in $\mathcal{M}(J)$, for some closed interval $J$. If $S$ is not semidiscrete, then Theorem~\ref{iab} tells us that $S$ is dense in $\mathcal{M}_0(J)$. However, Corollary~\ref{nps}\ref{wec} demonstrates that this can only be so if $S$ contains a two-generator semigroup that is not semidiscrete. Therefore, on the contrary, $S$ is semidiscrete.
\end{proof}

\section{Intersecting limit sets}

In this section we prove Theorem~\ref{zkj}, which says that if $S$ is a finitely-generated semidiscrete semigroup and $|\Lambda^-(S)|\neq 1$, then $\Lambda^+(S)=\Lambda^-(S)$ if and only if $S$ is a group. The next lemma is an important step in establishing this result.

\begin{lemma}\label{aif} 
Let $S$ be a nonelementary semigroup that satisfies $\Lambda^-(S)\subset \Lambda^+(S)$. Let  $f\in S$, let $p\in \Lambda^-(S)$, and let $U$ be a nontrivial open interval containing $p$. Then there exists an element $g$ of $S$ such that $fg$ is hyperbolic with attracting fixed point in $U$.
\end{lemma}
\begin{proof}
By shrinking $U$ if need be, we can assume that $\Lambda^-(S)$ intersects $\overline{\mathbb{R}}\setminus (\overline{U}\cup f^{-1}(\overline{U}))$. As a consequence, we see that $\Lambda^+(S)$ intersects  $\overline{\mathbb{R}}\setminus f^{-1}(\overline{U})$, and $\Lambda^-(S)$ intersects  $\overline{\mathbb{R}}\setminus \overline{U}$.

By Theorem~\ref{cca}, we can choose a hyperbolic element $h_1$ of $S$ such that $\beta_{h_1}\in U$ and $\alpha_{h_1}\in \overline{\mathbb{R}}\setminus f^{-1}(\overline{U})$. Let $n$ be a sufficiently large positive integer that $h_1^{-n}f^{-1}(\overline{U})$ is contained in $U$. Define $g_1=h_1^n$. Then $fg_1$ is hyperbolic (because it maps a closed interval into the interior of itself) and $\beta_{fg_1}\in U$. 

Let $V$ be an open interval containing $\beta_{fg_1}$ such that $fg_1(V)\subset U$. This interval intersects $\Lambda^+(S)$.  Using Theorem~\ref{cca} again, we can choose a hyperbolic element $h_2$ of $S$ such that $\alpha_{h_2}\in V$ and $\beta_{h_2}\in\overline{\mathbb{R}}\setminus \overline{U}$. Then $h_2^n(\overline{U})\subset V$ for a sufficiently large positive integer $n$. Let $g_2=h_2^n$. Then $fg_1g_2(\overline{U})\subset fg_1(V)\subset U$, so $fg_1g_2$ is hyperbolic and $\alpha_{fg_1g_2}\in U$. The lemma now follows on choosing $g=g_1g_2$. 
\end{proof}

Theorem~\ref{zkj} is an immediate consequence of the next, more general, theorem.

\begin{theorem} 
Let $S$ be a finitely-generated semidiscrete semigroup such that $|\Lambda^-(S)|\neq 1$. If $\Lambda^-(S)\subset \Lambda^+(S)$, then $S$ is a group.
\end{theorem}
\begin{proof}
By examining the cases in Theorem~\ref{mme}, we see that the theorem is true if $S$ is elementary (with the restriction $|\Lambda^-(S)|\neq 1$). Suppose instead that $S$ is nonelementary.

Let $\mathcal{F}$ be a finite generating set for $S$. Let $h\in \mathcal{F}$. Choose two distinct points $u$ and $v$ in $\Lambda^-(S)$. For each positive integer $n$, define $U_n$ to be the open interval of chordal radius $1/n$ that is centred on $u$, if $n$ is odd, and centred on $v$, if $n$ is even. Let $\Pi_n$ denote the closed hyperbolic half-plane with ideal boundary $U_n$. We will define a composition sequence $F_n=f_1\dotsb f_n$ generated by $S$ in a recursive fashion such that, for each $n\in\mathbb{N}$, we have $f_{2n-1}=h$ and $F_{2n}(i)\in \Pi_n$. 

Define $f_1=h$. By Lemma~\ref{aif}, there is an element $g$ of $S$ such that $f_1g$ is hyperbolic and $\alpha_{f_1g}\in U_1$. Therefore $(f_1g)^n(i)\in \Pi_1$, for a suitably large positive integer $n$. Define $f_2=g(f_1g)^{n-1}$. Then $F_2(i)\in \Pi_1$. 

Suppose now that we have constructed maps $f_1,\dots,f_{2n}$ in $S$ such that $f_{2k-1}=h$ and $F_{2k}(i)\in \Pi_k$, for $k=1,\dots,n$. Define $f_{2n+1}=h$. By Lemma~\ref{aif}, there is an element $g$ of $S$ such that $F_{2n+1}g$ is hyperbolic and $\alpha_{F_{2n+1}g}\in U_{n+1}$. Therefore $(F_{2n+1}g)^n(i)\in \Pi_{n+1}$, for large $n$. Define $f_{2n+2}=g(F_{2n+1}g)^{n-1}$. Then $F_{2n+2}(i)\in \Pi_n$. This completes the construction of the sequence $(F_n)$.

By writing each map $f_k$, for $k$ even, in terms of the generating set $\mathcal{F}$, we obtain a composition sequence $G_n=g_1\dotsb g_n$ generated by $\mathcal{F}$ that accumulates at $u$ and $v$ and is such that $g_k=h$ whenever $k$ is odd. If $S$ is not a cocompact Fuchsian group, then Corollary~\ref{axo} tells us that the maps $g_n$ lie in the group part of $S$, for all sufficiently large $n$. Therefore $h$ lies in the group part of $S$, so $h^{-1}\in S$. As we chose $h$ arbitrarily from $\mathcal{F}$, we conclude that $S$ must be a group.
\end{proof}

\section{Concluding remarks}\label{conclusion}

There are a wealth of open questions on semigroups, perhaps the most intriguing of which come from the theory of semigroups of complex M\"obius transformations. Some work has been carried out in this field already, in \cite{FrMaSt2012} (see also \cite{DaSiUr2015}). In a forthcoming paper, building on \cite{FrMaSt2012}, the first author proves that finitely-generated semigroups of complex M\"obius transformations with loxodromic generators and disjoint limit sets behave agreeably in several respects. For example, it will be shown that a small perturbation of a semigroup of this type gives rise to a small perturbation of the corresponding limit sets. Beyond this, the semigroups that are least well understood are those for which the forward and backward limit sets intersect; an example is illustrated in Figure~\ref{hxf}. 

\begin{figure}[ht]
\centering
\includegraphics[scale=0.08]{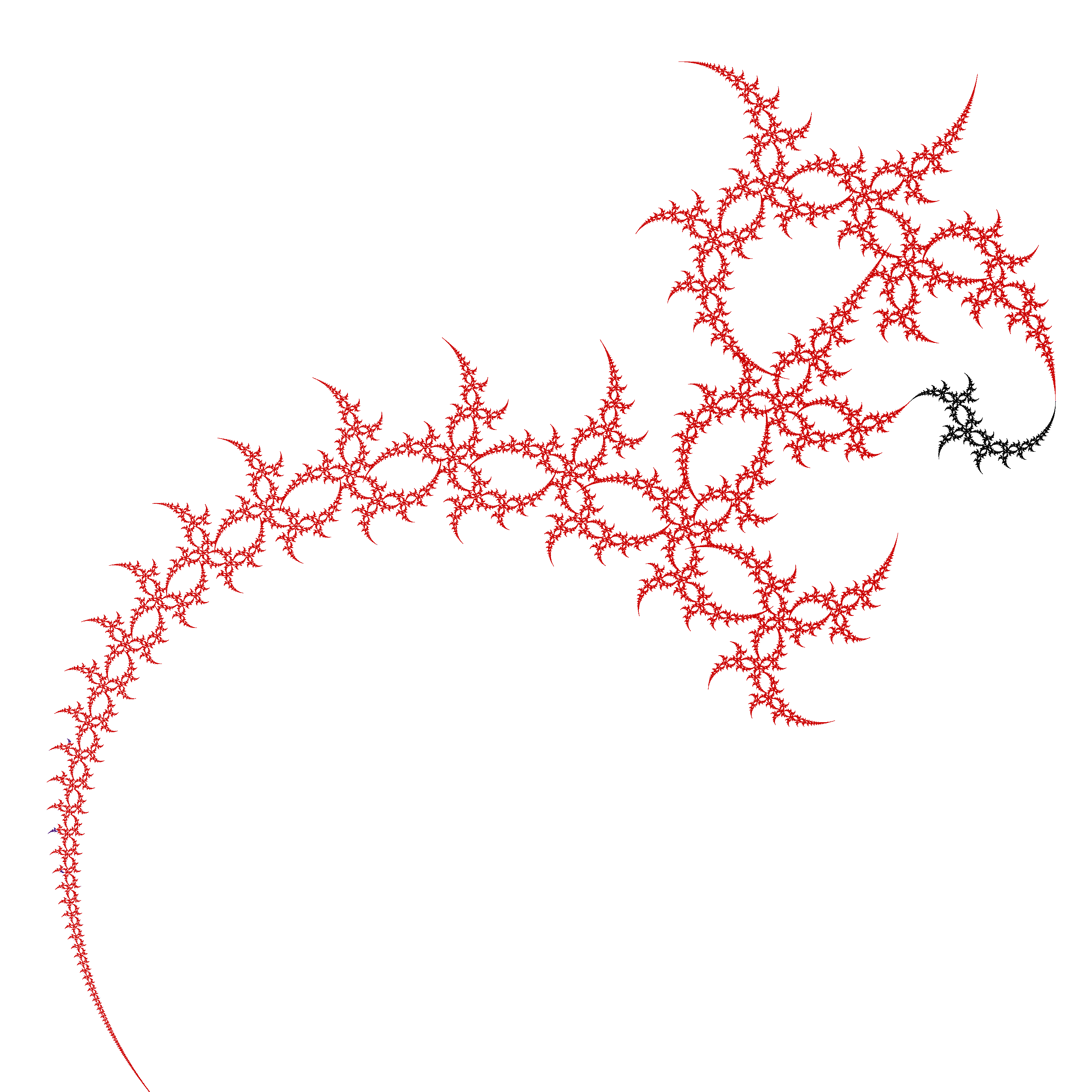}
\caption{Forward limit set (small, black shape) and backward limit set (large, lighter-coloured shape) of a conjugate of the semigroup $\langle z+a,z/(bz+1)\rangle$, where $a=0.5$ and $b=-1.1-1.1i$}
\label{hxf}
\end{figure}

We finish by supplying an example that answers Question~\ref{ccw} in the negative when $N=4$. We use the notation for the hyperbolic locus $\mathcal{H}$ and elliptic locus $\mathcal{E}$ defined in the introduction.

Let $F$ and $H$ be a pair of hyperbolic matrices in $\text{SL}(2,\mathbb{R})$ that are a standard group-generating set for a rank two Schottky group. Let $G=F^{-1}$ and $K=H^{-1}$. Then $(F,G,H,K)\notin \mathcal{E}$. We will prove that $(F,G,H,K)\notin \overline{\mathcal{H}}$. Suppose on the contrary that there is a sequence $(F_n,G_n,H_n,K_n)$ in $\mathcal{H}$ that converges to $(F,G,H,K)$. Observe that the matrix $F_nG_n$ is hyperbolic, for each $n$. By restricting to a subsequence if necessary, we can assume that there are points $a$ and $b$ in $\overline{\mathbb{R}}$ for which the attracting and repelling fixed points of $F_nG_n$ satisfy $\alpha_{F_nG_n}\to a$ and $\beta_{F_nG_n}\to b$ as $n\to\infty$.

Now choose a hyperbolic element $Q$ of the Schottky group $\langle F,G,H,K\rangle$ such that the axis of $Q$ meets neither $a$ nor $b$, and nor does it intersect the hyperbolic line from $a$ to $b$ (in the case $a\neq b$). We can choose $Q$ in such a way that $Q$ and $F_nG_n$ are antiparallel for all sufficiently large values of $n$. Since $Q\in\langle F,G,H,K\rangle$, we can find matrices $Q_n\in\langle F_n,G_n,H_n,K_n\rangle$ such that $Q_n\to Q$ as $n\to\infty$. Hence $Q_n$ and $F_nG_n$ are antiparallel for sufficiently large $n$. It follows from J{\o}rgensen's inequality for semigroups, Theorem~\ref{fff}, that
\[
\lvert\Tr(F_nG_n)^2-4\rvert+\lvert\Tr[F_nG_n,Q_n]-2\rvert\geq 1
\]
for all sufficiently large $n$. However, $(F_nG_n)$ converges to the identity matrix as $n\to\infty$, so the left-hand side of this inequality converges to $0$, a contradiction. Thus, contrary to our assumption, there is no such sequence $(F_n,G_n,H_n,K_n)$, and $\overline{\mathcal{H}}\neq \mathcal{E}^c$. Bochi pointed out to us that the inequality of \cite[Lemma~4.8]{AvBoYo2010} can be used instead of J{\o}rgensen's inequality in the last step of the argument; the two inequalities are closely related.

Another example that provides a negative answer to Question~\ref{ccw} has been offered by Argyrios Christodoulou (to appear). This example uses generators of an elementary semigroup that contains neither elliptic elements nor the identity map but contains the identity map in its closure. It remains to be seen how to modify the (incorrect) assertion that $\overline{\mathcal{H}}= \mathcal{E}^c$ to take account of these two types of examples. 

Interested readers should  consult \cite{AvBoYo2010,Yo2004} for a number of related questions.                                                           

\subsection*{Acknowledgements}

The authors thank Jairo Bochi for bringing to our attention Question~\ref{ccw} (from \cite{AvBoYo2010,Yo2004}). We also thank Edward Crane for many helpful comments.

\end{document}